\documentclass[10pt,reqno,a4paper]{amsart}
\usepackage{euscript}
\usepackage{amssymb}
\usepackage{breqn}
\usepackage{subfigure,graphicx}
\usepackage{overpic}
\usepackage{color}
\usepackage{caption}%
\usepackage[font=small,labelfont=bf]{caption}
\usepackage[normalem]{ulem}

\usepackage{float}

\DeclareCaptionFont{tiny}{\tiny}

\newcommand{\R}{\ensuremath{\mathbb{R}}}

\newcommand{\CC}{\mathcal{C}}

\newcommand{\CF}{\ensuremath{\mathcal{F}}}

\newcommand{\CO}{\ensuremath{\mathcal{O}}}
\newcommand{\CZ}{\ensuremath{\mathcal{Z}}}
\newcommand{\ov}{\overline}

\newcommand{\la}{\lambda}
\newcommand{\ga}{\gamma}
\newcommand{\G}{\Gamma}
\newcommand{\T}{\theta}
\newcommand{\f}{\varphi}
\newcommand{\al}{\alpha}

\newcommand{\s}{\ensuremath{\mathbb{S}}}

\newcommand{\U}{\ensuremath{\mathcal{U}}}

\newcommand{\de}{\delta}
\newcommand{\De}{\Delta}

\newcommand{\bxi}{\boldsymbol \xi}

\newcommand{\bx}{{\bf x}}
\newcommand{\by}{{\bf y}}
\newcommand{\bz}{{\bf z}}
\newcommand{\bu}{{\bf u}}
\newcommand{\bv}{{\bf v}}
\newcommand{\bw}{{\bf w}}
\newcommand{\bg}{{\bf g}}
\newcommand{\bF}{{\bf F}}
\def\p{\partial}
\def\e{\varepsilon}

\newtheorem {theorem} {Theorem} 
\newtheorem {proposition} [theorem] {Proposition}

\newtheorem {remark} {Remark}

\newtheorem {mtheorem} {Theorem}

\begin{document}

\title[Periodic solutions and invariant torus]
{Periodic solutions and invariant torus \\ in the R\"ossler System}

\author[M.R. C\^{a}ndido, C. Valls and D. D. Novaes]
{Murilo R. C\^andido$^1$, Douglas D. Novaes$^{1}$, and Claudia Valls$^2$ }

\address{$^1$ Departamento de Matem\'{a}tica, Universidade
Estadual de Campinas, Rua S\'{e}rgio Buarque de Holanda, 651, Cidade Universit\'{a}ria Zeferino Vaz, 13083-859, Campinas, SP,
Brazil}
\email{candidomr@ime.unicamp.br}
\email{ddnovaes@unicamp.br}

\address{$^2$ Departamento de Matem\'atica, Instituto Superior
T\'ecnico, Universidade de Lisboa, Av. Rovisco Pais 1049-001,
Lisboa, Portugal} \email{cvalls@math.ist.utl.pt}



\date{}
\dedicatory{}

\maketitle

\begin{abstract}
The R\"ossler System is characterized by a three-parameter family of quadratic 3D vector fields. There exist two one-parameter families of R\"ossler Systems exhibiting a zero-Hopf equilibrium. For R\"ossler Systems near to one of these families, we provide generic conditions ensuring the existence of a torus bifurcation. In this case, the torus surrounds a periodic solution that bifurcates from the zero-Hopf equilibrium. 
For R\"ossler Systems near to the other family, we provide generic conditions for the existence of a periodic solution bifurcating from the zero-Hopf equilibrium. This improves currently known results regarding periodic solutions for such a family. In addition, the stability properties of the periodic solutions and invariant torus are analysed.
\end{abstract}

\allowdisplaybreaks
\section{Introduction and statement of the main results}

The {\it R\"{o}ssler  System} was introduced in 1976 by R\"ossler \cite{Rossler76} as a prototype of a simple autonomous differential system behaving chaotically for some values of the parameters:
\begin{equation}\label{sr}
\begin{array}{l}
\dot{x}=-y-z,\vspace{0.2cm}\\
\dot{y}= x+ a y,\vspace{0.2cm}\\
\dot{z}= b x-c z + x z.
\end{array}
\end{equation}
By simple we mean low dimensional, few parameters, and only one non-linear term. Originally, this system was conceived for helping to understand the chaotic properties of some differential models of chemical reactions \cite{Rossler76b,Rossler77,Rossler78, Rossler79}. Since then, the chaotic behavior of the R\"ossler System has been addressed in several works. We may cite, for instance,  \cite{BBS,WSB,Z} and the references therein. 

Detecting periodic solutions in the R\"ossler System \eqref{sr} has also been a subject of interest  for many authors. A brief summary of these results can be found in \cite{Ll}, which we shall subsequently complement. In 1984, Glendinning and Sparrow \cite{GS} showed the existence of periodic solutions of the R\"ossler System near some homoclinic solutions. In 1995, Krishchenko \cite{Kr} proved that all periodic solutions of the R\"ossler System must lie in a specific bounded domain. In the same year, Magnitskii \cite{Ma} obtained asymptotic formulae for the amplitude and period of the periodic
solutions arising from Hopf bifurcations in the R\"ossler System.  In 1999, Terekhin and Panfilova \cite{TePa} provided sufficient conditions for the existence of periodic
solutions near the equilibria of the R\"ossler System. In 2000 and 2003, Pilarczyk \cite{pi1, Pi} used the Conley Index Theory to provide a computer-assisted proof that several periodic solutions exist in the R\"ossler System  for some parameter values. In 2006, Galias  \cite{ga06} developed a numerical method to study short-period solutions and applied it to the R\"ossler System. In 2007, Algaba et al. \cite{Al} studied the merging of the periodic solutions that appeared in resonances while also demonstrating the existence of two types of Takens-Bogdanov bifurcations of periodic solutions. In 2009, Wilczak and Zgliczy{\'n}ski \cite{WZ09} proved the existence of two period-doubling bifurcations connected by a branch of period two solutions for a specific range of the parameters of the R\"ossler System.

The {\it Averaging Theory} is a classical method and one of the main tools for detecting periodic solutions in regularly perturbed non-autonomous differential systems. Roughly speaking, this method provides a sequence of functions, $\bg_i$, each one called {\it $i$-th order averaged function}, for which their simple zeros correspond to isolated periodic solutions of the differential system. In 2007, Llibre at al. \cite{LlBu} used the first-order averaging method to study Hopf bifurcations in the R\"ossler System. More recently, in 2014,  Llibre \cite{Ll} used the first-order averaging method to study periodic solutions bifurcating from zero-Hopf equilibria of the R\"ossler System.  Here, a zero-Hopf equilibrium is an equilibrium of the differential system where the Jacobian matrix has a zero eigenvalue and a pair of purely imaginary conjugate eigenvalues.

In our study, we shall apply some recent developments of the Averaging Theory to improve the results of \cite{Ll} in two directions:

\smallskip

\noindent {\bf Case  A}:  Firstly,  for $(a,b,c)=(\ov a, 1,\ov a)$, with $\ov a\in (-\sqrt{2},\sqrt{2})\setminus\{0\}$, one can see that the R\"{o}ssler System \eqref{sr} has a zero-Hopf equilibrium at the origin. In \cite{Ll}, assuming that the parameter vector $(a,b,c)$ is $\e$-close to $(\ov a, 1,\ov a)$, that is, $(a,b,c)=(\ov a, 1,\ov a)+\CO(\e),$  the existence of a periodic solution bifurcating from the zero-Hopf equilibrium at the origin for $\e=0$ has already been proven (see \cite[Theorem 2]{Ll}). Here, in our first main result (Theorem \ref{t1}), we provide  the existence of 
an invariant torus, caused by a Neimark-Sacker bifurcation, situated around this periodic solution (see Figures \ref{fig3a} and \ref{fig3b}). This kind of bifurcation had been previously indicated for the  R\"ossler System \cite{Al,BBAS}. Nevertheless, to the best of our knowledge, this is the first time that analytic generic conditions are provided ensuring  the existence of an invariant torus bifurcating from a zero-Hopf equilibrium in the R\"ossler System. 

\smallskip

\noindent {\bf Case  B}: Secondly, for $(a,b,c)=(0, \ov b,0)$, with $\ov b\in (-1,\infty)$,  again one can see that the R\"{o}ssler System \eqref{sr} has a zero-Hopf equilibrium at the origin. In \cite{Ll}, assuming that the parameter vector $(a,b,c)$ is $\e$-close to $(0, \ov b,0)$, that is, $(a,b,c)=(0, \ov b,0)+\CO(\e),$  the first-order averaging method has already been proven to not be able to detect any periodic solution bifurcating from the zero-Hopf equilibrium at the origin for $\e=0$ (see \cite[Theorem 3]{Ll}). This essentially means that the first-order averaged function, associated with the R\"{o}ssler System, does not have simple zeros. However, in general, it does not imply that such a bifurcation is not happening. Roughly speaking, in the research literature, the next natural step would usually consist in assuming some constrains on the first-order approximation (in $\e$) of the parameters such that the first-order averaged function vanishes identically, and then computing the simple zeros of the second-order averaged function. This method can be implemented at any order of perturbation. However, we shall see that this procedure fails in providing periodic solutions at least up to fifth-order (see Section \ref{sec:SA}). Here, in our second main contribution (Theorem \ref{t2}),  we shall apply a recent result on averaging theory (see \cite{CLN}), based on the Lyapunov-Schmidt reduction, which will allow us to use, simultaneously, the second- and third-order averaged functions for detecting  a periodic solution bifurcating from this zero-Hopf equilibrium (see Figure \ref{fig1}). In addition, we shall use the forth- and fifth- averaged functions to study the stability of this periodic solution.

\smallskip

This paper is organized as follows. In Section \ref{sec:ps}, we first introduce the bifurcation theory to study the existence of periodic solutions when the first-order averaged function is non-vanishing but can have, eventually, non-isolated zeros (see \cite{CLN}). Then, we apply this theory to study the existence of periodic solutions for {\bf Case A} and {\bf Case B} of the R\"ossler System \eqref{sr}. The stability properties of these periodic solutions are studied in Section \ref{sec:stab}, using mainly the theory of $k$-determined hyperbolicity for perturbed matrices (see \cite{Mu}). In Section \ref{sec:it}, we first introduce the recently developed theory for detecting invariant tori through the averaging theory (see \cite{ITCanNov2018}). Then, we apply this theory to study the existence of an invariant torus for {\bf Case A}  of the R\"ossler System \eqref{sr}.  In Section \ref{sec:ex}, we provide numerical examples for which our main results apply. Finally, a discussion of our main contributions is provided in Section \ref{sec:dis}.  

\bigskip

We summarize our main results as follows:

\smallskip

For {\bf Case A}, we consider the parameter vector $(a,b,c)$ of the R\"{o}ssler system \eqref{sr} $\e$-close to $(\ov a,1,\ov a).$ More specifically, we assume that 
\begin{equation}\label{pvCa}
(a, b, c)=\left(\ov{a}+\e \al_1+\e^2 \al_2 , 1+\e \beta_1+\e^2 \beta_2, \ov{a}+\e\gamma_1+\e^2\gamma_2\right)+\CO(\e^3),
\end{equation}
 with  $\ov{a}\in (-\sqrt{2}, \sqrt{2})\setminus \{0\}$ and $\e,\al_i,\beta_i,\gamma_i\in\R,$ for $i=1,2.$ Also, define
\begin{equation}\label{CA}
\begin{array}{rl}
d_0=&\left(\alpha_1-\gamma_1+\beta_1  \ov{a} \left(\ov{a}^2-1\right)\right) \left(\alpha_1 \left(\ov{a}^2-1\right)+\ov{a} (\beta_1 -\ov{a} \gamma_1)+\gamma_1\right),\vspace{0.2cm}\\
d_1=& \gamma_1-\al_1+\beta_1  {\ov a},\,\,\text{and}\vspace{0.2cm}\\
 \ell_{1}=&12\pi \ov{a}^4 \left(\ov{a}^4-16\right)-
 
 \ov{a} \left(4 \ov{a}^8-12 \ov{a}^6+193 \ov{a}^4-640 \ov{a}^2-144\right) \sqrt{2-\ov{a}^2} 
 .\end{array}
\end{equation}

Notice that the parameters above, $d_0, d_1,$ and $\ell_1,$ do not depend on $\e.$
\begin{mtheorem}\label{t1}
Let $(a, b, c)$ be given by \eqref{pvCa}.
\begin{itemize}
\item[(i)] If $d_0>0$, then for $|\e|\neq0$ sufficiently small the R\"{o}ssler System \eqref{sr} admits a periodic solution $\varphi(t,\e)$ satisfying $\varphi(t,\e)\to (0,0,0)$ when $\e\to 0$. Moreover, for $\e>0$, such a periodic solution is asymptotically stable (resp. unstable) provided that $d_1> 0$  (resp. $d_1<0$). Denote $\varphi(t,\gamma_1,\e)=\varphi(t,\e).$
\item[(ii)]  In addition, if $\ell_1\neq0$, then there exist a smooth curve $\gamma(\e)$, defined for $\e>0$ sufficiently small and satisfying $\gamma(\e)=\ov\gamma_1+\CO(\e)$ with $\ov\gamma_1=\al_1-\ov a \beta_1$, and intervals $J_{\e}$ containing $\gamma(\e)$ such that a unique invariant torus bifurcates from the periodic solution $\varphi(t,\gamma(\e),\e)$ as $\gamma_1$ passes through $\gamma(\e).$ Such a torus exists whenever $\gamma_1\in J_{\e}$ and $\ell_1(\gamma_1-\gamma(\e))>0,$ and surrounds the periodic solution $\f(t,\gamma_1,\e).$ In addition,  if $\ell_1>0$ (resp. $\ell_1<0$) the torus is unstable (resp. asymptotically stable), whereas the periodic solution $\varphi(t,\gamma_1,\e)$ is asymptotically stable (resp. unstable).
\end{itemize}
\end{mtheorem}
The proof of Theorem \ref{t1} will be split into several propositions in the following sections. Statement (i) will follow from Propositions \ref{p1} and \ref{p3} of Sections \ref{sec:ps} and \ref{sec:stab}, respectively, and Statement (ii) will follow from Proposition \ref{p5} of Section \ref{sec:it}.

\bigskip

For {\bf Case B}, we consider the parameter vector $(a,b,c)$ of the R\"{o}ssler system \eqref{sr} $\e$-close to $(0, \ov b,0),$ with $\ov b\in (-1,\infty).$ More specifically, taking $\ov b=\omega^2-1,$ we assume that
\begin{equation}\label{pvCb}
(a,b,c)=(\al(\e), \omega^2-1+\beta(\e), \gamma(\e)), \text{ with }\omega>0,\, \omega \notin\{1, \sqrt{2}\},
\end{equation}
and
\[
\al(\e)=\sum_{i=1}^5\e^i\al_i+\CO(\e^6), \,\, \beta(\e)=\sum_{i=1}^5\e^i\beta_i+\CO(\e^6),\,\, \text{and}\,\, \gamma(\e)=\sum_{i=1}^5\e^i\gamma_i+\CO(\e^6).
\] 
with $\e,\al_i,\beta_i,\gamma_i\in\R,$ for $i=1,\ldots, 5.$
 Also, define
\begin{equation}\label{CB}
\begin{array}{l}
\lambda_1=\gamma_1 \left(\omega ^2-2\right),\quad \lambda_{2}=\gamma_1 \left(1-\omega ^2\right),\quad \text{and}\vspace{0.2cm}\\
 \delta=\dfrac{\beta_1 \gamma_2+\beta_2 \gamma_1+\gamma_1^3 \left(\omega ^4-3 \omega ^2+2\right)+\gamma_3 \left(\omega ^2-1\right)-\alpha_3}{\gamma_1 \left(1-\omega ^2\right)}.
\end{array}
\end{equation}
Notice that the parameters above, $\lambda_1, \lambda_2,$ and $\delta,$ do not depend on $\e.$
\begin{mtheorem}\label{t2}
Let $(a,b,c)$ be given by \eqref{pvCb}.
Suppose that $\al_1=\gamma_1(\omega^2-1)$,  $\al_2=\beta_1 \gamma_1+\gamma_2 (\omega^2-1)$, and $\delta>0$. Then, for $|\e|\neq0$ sufficiently small, the R\"{o}ssler System \eqref{sr} has a periodic solution $\varphi(t,\e)$  satisfying $\varphi(t,\e)\to (0,0,0)$.  Moreover, for $\e>0$, such periodic solution is asymptotically stable (resp. unstable) provided that  $\lambda_1<0$ and  $\lambda_2<0$ (resp. $\lambda_1>0$ or $\lambda_2>$0).
\end{mtheorem}

The proof of Theorem \ref{t2} will follow from Propositions \ref{p2} and \ref{p4} of Sections \ref{sec:ps} and \ref{sec:stab}, respectively.

\section{Bifurcation of periodic solutions}\label{sec:ps}

The averaging method is one of the main tools for detecting periodic solutions in regularly perturbed non-autonomous differential systems. This method has been generalized in several directions. In Section \ref{sec:at}, we introduce the classical version of the averaging theorem (Theorem \ref{CAT}) as well as its recent generalization (Theorem \ref{te}) based on the Lyapunov-Schmidt reduction. Then, in Sections \ref{sec:psca} and \ref{sec:pscb}, these theorems are applied to prove the existence of periodic solutions for {\bf Case A} and {\bf Case B} of the R\"{o}ssler System \eqref{sr}, respectively. Additionally, in Section \ref{sec:SA}, we show that in {\bf Case B}  the usual recursive method of applying the higher order averaging method fails in detecting periodic solutions of the R\"{o}ssler System up to order five. This emphasizes the importance of the method developed in \cite{CLN}.

\subsection{Averaging Theory and Bifurcation Functions}\label{sec:at} The averaging theory provides sufficient conditions for the existence of periodic solutions of non-autonomous differential systems written in the following {\it standard form}:
\begin{equation}\label{tm13}
\dot {\bf x}=\sum_{i=1}^k \e^i {\bf F}_i(t,{\bf
x})
+\e^{k+1} \widetilde{{\bf
F}}(t,{\bf x},\e), \quad  (t,{\bf x},\e)\in\R\times\Omega\times(-\e_0,\e_0),
\end{equation}
where $\Omega$ is an open bounded subset of $\mathbb{R}^2$ and  $\e_0$ is a small positive real number. It is assumed that $\bF_i,$ $i=1,\ldots,k,$ and $\widetilde\bF$ are sufficiently smooth functions $T$-periodic in the variable $t.$ The periodicity of system \eqref{tm13} allow us to see it defined in the cylinder $(t,\bx) \in \s^1\times \Omega,$ where $\s^1\equiv \R/T\mathbb{Z}.$
 
This method provides a sequence of functions $\bg_i$, $i=1,\ldots,k$, such that their simple zeros lead to isolated $T$-periodic solutions of system \eqref{tm13}.   These functions are obtained as follows. The solution $\bx(t,\bz,\e)$ of \eqref{tm13}, satisfying $\bx(0,\bz,\e)=\bz$, can be written as
\begin{equation*}\label{tri}
\bx(t,\bz,\e)=\bz+\sum_{i=1}^k\e^i \dfrac{\by_i(t,\bz)}{i!}+\CO(\e^{k+1}),
\end{equation*}
where the expressions for $\by_i$ are obtained by solving recursively the system of equations obtained from \eqref{tm13} (see \cite[Lemma 5]{CLN}). 
Hence, the Poincar\'{e} map $\Pi(\bz,\e)=\bx(T,\bz,\e)$ can be written as
\begin{equation}\label{quad}
\Pi(\bz,\e)=\bz+\sum_{i=1}^k\e^i\bg_i(\bz) +\CO(\e^{k+1}),
\end{equation}
where 
\begin{equation}\label{avf}
\bg_i(\bz)=\dfrac{\by_i(T,\bz)}{i!}
\end{equation} 
is called \textit{averaged functions of order $i$}  of system \eqref{tm13}.
Define $\bg_0=0.$ If for some $m\in\{1,2,\ldots,k\}$ we have that $\bg_0=  \dots= \bg_{m-1}= 0$ and $\bg_m\neq0$, then a simple zero of $\bg_m(\bz)$ provides a branch of fixed points $\bz(\e)$ for the map \eqref{quad}, that is, $\Pi(\bz(\e),\e)=\bz(\e)$. In turns, $\bx(t,\bz(\e),\e)$ corresponds to a branch of isolated $T$-periodic solutions of system \eqref{tm13}. This will be the content of Theorem \ref{CAT}.

For the readers' convenience, we present the expressions of the functions $\by_i,$ for $i=1,\ldots,5$. For the general expressions, the reader is addressed to \cite{CLN,nov17}. Consider the vector $\by=(y^1,\dots, y^n)\in\R^n$, we denote $\by^m=\big(\by, \cdots, \by \big)\in \R^{m n}$.  In the following expressions, we represent the $l-$Frechet derivative of $\bF_i(t,\bx)$ applied to a ``product'' of $l$ vectors as the multilinear map:
$$
\begin{array}{rr}
\displaystyle\dfrac{\p^l \bF_i}{\p \bx^l}(t,\bx)\bigodot^l_{j=1}\by_j=& \displaystyle\Bigg(\sum^n_{i_1,\dots, i_l=1} \dfrac{\partial^l F_i^1}{\partial x_{i_1}\dots \partial x_{i_l}}(t,\bx)y_1^{i_1}\dots y_l^{i_l}, \dots, \vspace{0.2cm}\\
& \displaystyle\sum^n_{i_1,\dots, i_l=1} \dfrac{\partial^l F_i^n}{\partial x_{i_1}\dots \partial x_{i_l}}(t,\bx)y_1^{i_1}\dots y_l^{i_l}\Bigg).
\end{array}
$$
For $i=1,\ldots,5$, we define the \textit{averaged functions of order $i$}  of system \eqref{tm13} as
\begin{equation}\label{avf}
\bg_i(\bz)=\dfrac{\by_i(T,\bz)}{i!},
\end{equation}
where
\begin{align*}
\by_1(t,\bz)=&\int_0^t \bF_1(\tau,\bz)\mathrm{d}\tau,\\
\by_2(t,\bz)=&\int_0^t
2\bF_2(\tau,\bz)+2\dfrac{\p
\bF_1}{\p
\bx}(\tau,\bz)\by_1(\tau,\bz) \mathrm{d}\tau,\\
\by_3(t,\bz)=&\int_0^t
6\bF_3(\tau,\bz)+6\dfrac{\p
\bF_2}{\p
\bx}(\tau,\bz)\by_1(\tau,\bz)\\
&+3\dfrac{\p^2 \bF_1}{\p
\bx^2}(\tau,\bz)\by_1(\tau,\bz)^2+3\dfrac{\p \bF_1}{\p
\bx}(\tau,\bz)\by_2(\tau,\bz)\mathrm{d}\tau,\\
\by_4(t,\bz)=&\int_0^t
24\bF_4(\tau,\bz)+24\dfrac{\p \bF_3}{\p
\bx}(\tau,\bz)\by_1(\tau,\bz) +12\dfrac{\p^2 \bF_2}{\p
\bx^2}(\tau,\bz)\by_1(\tau,\bz)^2+\\
&12\dfrac{\p \bF_2}{\p
\bx}(\tau,\bz)\by_2(\tau,\bz)+12\dfrac{\p^2 \bF_1}{\p \bx^2}(\tau,\bz)\by_1(\tau,\bz)\odot
\by_2(\tau,\bz)\\
&+4\dfrac{\p^3 \bF_1}{\p
\bx^3}(\tau,\bz)\by_1(\tau,\bz)^3+4\dfrac{\p \bF_1}{\p \bx}(\tau,\bz)\by_3(\tau,\bz)\mathrm{d}\tau,\\
\by_5(t,\bz)=&\int_0^t
120\bF_5(\tau,\bz)+120\dfrac{\p \bF_4}{\p
\bx}(\tau,\bz)\by_1(\tau,\bz) \vspace{0.3cm}\\
&+60\dfrac{\p^2 \bF_3}{\p
\bx^2}(\tau,\bz)\by_1(\tau,\bz)^2+60\dfrac{\p \bF_3}{\p
\bx}(\tau,\bz)\by_2(\tau,\bz)\vspace{0.3cm}\\
&+60\dfrac{\p^2 \bF_2}{\p \bx^2}(\tau,\bz)\by_1(\tau,\bz)\odot
\by_2(\tau,\bz)+20\dfrac{\p^3 \bF_2}{\p
\bx^3}(\tau,\bz)\by_1(\tau,\bz)^3\vspace{0.3cm}\\
&+20\dfrac{\p \bF_2}{\p \bx}(\tau,\bz)\by_3(\tau,\bz)+20\dfrac{\p^2
\bF_1}{\p \bx^2}(\tau,\bz)\by_1(\tau,\bz)\odot \by_3(\tau,\bz)\vspace{0.3cm}\\
&+15\dfrac{\p^2 \bF_1}{\p
\bx^2}(\tau,\bz)\by_2(\tau,\bz)^2+30\dfrac{\p^3 \bF_1}{\p
\bx^3}(\tau,\bz)\by_1(\tau,\bz)^2\odot \by_2(\tau,\bz) \vspace{0.3cm}\\
&+5\dfrac{\p^4 \bF_1}{\p
\bx^4}(\tau,\bz)\by_1(\tau,\bz)^4+5\dfrac{\p \bF_1}{\p
\bx}(\tau,\bz)\by_4(\tau,\bz)\vspace{0.3cm}\mathrm{d}\tau.
\end{align*}
 With these functions the classical averaging method for finding periodic solutions can be summarized by the following theorem, which relates zeros of the first non-vanishing averaged function to the existence of periodic solutions of the non-autonomous differential system \eqref{tm13}.

\begin{theorem}[\cite{DLT}]\label{CAT}
Assume that, for some $m\in\{1,\ldots,k\}$, $\bg_0=\cdots\bg_{m-1}=0$ and $\bg_m\neq0$. If there exists $\bz^*\in \Omega$
such that $\bg_m(\bz^*)=0$ and $| D\bg_m(\bz^*)|\neq 0$, then for $|\e|\neq0$ sufficiently small there exists an isolated $T$-periodic solution
$\f(t,\e)$ of system \eqref{tm13} such that $\f(0,0)=\bz^*$.
\end{theorem}

The previous result  says that a simple zero of the first non-vanishing averaged function $\bg_m$ corresponds to a periodic solution of system \eqref{tm13}. In the case that the zero is not simple but isolated, one can still use some topological version of Theorem \ref{CAT} to ensure the existence of  periodic solutions (see, for instance, \cite{B04,DLT}). However, it cannot be used when the zero is not isolated.
This problem has been addressed in \cite{CLN} and we present its main result in the sequel. 

Let $\pi:\R^m\times\R^{n-m} \rightarrow\R^m$ and
$\pi^{\perp}:\R^m\times \R^{n-m} \rightarrow\R^{n-m}$  be the
projections onto the first $m$ coordinates and onto the last $n-m$
coordinates, respectively. Denote $\bz\in \Omega$ as $\bz=(u,v)\in \R^m\times\R^{n-m}$. Assume that the first-order averaged function $\bg_1$ vanishes on the set
\begin{equation*}\label{grp}
\CZ=\lbrace \bz_u =(u,\mathcal{B}(u)):u \in \ov V  \rbrace\subset U,
\end{equation*}
where $m<n$ are positive integers, $V$ is an open bounded subset of $\R^m,$ and  $\mathcal{B}:\ov V\rightarrow \R^{n-m}$ is a  $\CC^k$ function. Thus, $\CZ$ is a set of non-isolated zeros of $\bg_1$ and, consequently, Theorem \ref{CAT} cannot be applied. Nevertheless, as performed in \cite{CLN}, the Lyapunov-Schmidt reduction can be used to obtain sufficient conditions for the existence of isolated $T$-periodic solutions bifurcating from $\CZ$ as follows. First, we notice that the equation $\bz=\Pi(\bz,\e)$ is equivalente to the following system of equations
\begin{equation}\label{SE}
\left\{\begin{array}{l}
0=u-\pi \Pi(u,v,\e),\vspace{0.2cm}\\
0=v-\pi^{\perp} \Pi(u,v,\e).\\
\end{array}\right.
\end{equation}
Under convenient assumptions, the {\it implicit function theorem} can be used to find a function $\ov{\mathcal{B}}(u,\e),$ satisfying $\ov{\mathcal{B}}(u,0)=\mathcal{B}(u),$ which solves the second line of system \eqref{SE}, that is,  $\ov{\mathcal{B}}(u,\e)=\pi^{\perp} \Pi(u,\ov{\mathcal{B}}(u,\e),\e)$. Then, substituting $v=\ov{\mathcal{B}}(u,\e)$ into the first line of system \eqref{SE}, we obtain a single equation to be solved, namely $\CF(u,\e)=u- \pi \Pi(u,\ov{\mathcal{B}}(u,\e),\e)=0.$ Expanding $\CF(u,\e)$ around $\e=0,$ we get the bifurcation functions $f_i$, for $i=1,\ldots,k-1,$
\begin{align*}
\CF(u,\e)=\sum_{i=1}^{k-1} \e^i f_i(u)+\CO(\e^{k+1}),
\end{align*}
which will be given in terms of the derivatives $\gamma_j(u)=(\partial^j\ov{\mathcal{B}}/\partial \e^j)(u,0),$ $j=1,\ldots,i,$ and averaged functions $\bg_j,$ $j=1,\ldots,i+1.$ Denote $f_0=0.$ Notice that, if for some $m\in\{1,\ldots,k\}$ we have $f_0=\cdots= f_{m-1}=0$ and $f_m\neq0$, then a simple zero of $f_m(u)$ provides a branch of zeros $u(\e)$ of $\CF$, that is, $\CF(u(\e),\e)=0.$ Consequently, $\bz(\e)=\big(u(\e),\ov{\mathcal{B}}(u(\e),\e)\big)$ is a branch of solutions of system \eqref{SE}, that is, fixed points for the map \eqref{quad}. Again, $\bx(t,\bz(\e),\e)$ corresponds to a branch of isolated $T$-periodic solutions of system \eqref{tm13}. This will be the content of Theorem \ref{te}.

In what follows, we present the expressions of the functions $\gamma_i$ and $f_i$ for $i=1,\ldots, 4$. Denote
\[
D \bg_1({\bf z}_u)=\begin{pmatrix}
\Lambda_u & \Gamma_u \\
B_u & \Delta_u
\end{pmatrix} 
\]
where $\Lambda_u=\partial_a\pi \bg_1(z_u)$, $\Gamma_u=\partial_b\pi
\bg_1(z_u)$, $B_u=\partial_a\pi^\perp \bg_1(z_u)$ and
$\Delta_u=\partial_b\pi^\perp \bg_1(z_u)$. 
The  {\it bifurcation function of order $i$} $f_i,$ for $i=1,\ldots 4,$ are defined as
\begin{equation}\label{bif1}
\begin{array}{rl}
\gamma_1(u)=&-\Delta_u^{-1}\pi^\perp \bg_{2}(\bz_u), \vspace{0.3cm}\\
f_1(u)=&\Gamma_u\gamma_1(u)+\pi {\bf g}_{2}({\bf z}_u),\\
\end{array}
\end{equation}

\begin{equation}\label{bif2}
\begin{array}{rl}
\gamma_2(u)=&-\Delta_u^{-1}\Bigg(\dfrac{\p^2\pi^\perp
\bg_1}{\partial b^2}(\bz_u)\gamma_1(u)^2
+2\dfrac{\partial\pi^\perp \bg_{2}}{\partial
b}(\bz_u)\gamma_1(u)+2\pi^\perp \bg_{3}(u)
\Bigg),\vspace{0.3cm}\\
f_2(u)=&\dfrac{1}{2}\Gamma_u\gamma_2(u)+\dfrac{1}{2}\dfrac{\p^2\pi
\bg_1}{\p b^2}(\bz_u)\gamma_1(u)^2+\dfrac{\partial\pi
\bg_{2}}{\partial
b}(z_u)\gamma_1(u)+\pi \bg_{3}(\bz_u), \vspace{0.3cm}\\
\end{array}
\end{equation}

\begin{equation}\label{bif3}
\begin{array}{rl}
\ga_3(u)=&-\De_u^{-1}\Bigg(\dfrac{\p^3\pi^\perp\bg_1}{\p
b^3}(z_u)\ga_1(u)^3 +3\dfrac{\p^2\pi^\perp\bg_1}{\p
b^2}(z_u)\ga_1(u)\odot \ga_2(u) \vspace{0.3cm}\\
&+3\dfrac{\p^2\pi^\perp \bg_{2}}{\p
b^2}(z_u)\ga_1(u)^2+2\dfrac{\p\pi^\perp \bg_{2}}{\p
b}(z_u)\ga_2(u)+6\dfrac{\p\pi^\perp \bg_{3}}{\p
b}(z_u)\ga_1(u)\vspace{0.3cm}\\
&+6\pi^\perp \bg_{4}(u) \Bigg),\vspace{0.3cm}\\
f_3(u)=&\dfrac{1}{6}\G_u\ga_3(u)+\dfrac{1}{6}\dfrac{\p^3\pi
\bg_1}{\p b^3}(z_u)\ga_1(u)^3+\dfrac{1}{2}\dfrac{\p^2\pi\bg_1}{\p
b^2}(z_u)\ga_1(u)\odot \ga_2(u)\vspace{0.3cm}\\
&+\dfrac{1}{2}\dfrac{\p^2\pi \bg_{2}}{\p
b^2}(z_u)\ga_1(u)^2+\dfrac{1}{2}\dfrac{\p\pi \bg_{2}}{\p
b}(z_u)\ga_2(u)+\dfrac{\p\pi \bg_{3}}{\p b}(z_u)\ga_1(u)\vspace{0.3cm}\\
&+\pi \bg_{4}(z_u), \vspace{0.3cm}\\
\end{array}
\end{equation}

\begin{equation}\label{bif4}
\begin{array}{rl}
\ga_4(u)=&-\De_u^{-1}\Bigg(\dfrac{\p^4\pi^\perp\bg_1}{\p
b^4}(z_u)\ga_1(u)^4 +3\dfrac{\p^2\pi^\perp\bg_1}{\p
b^2}(z_u)\ga_2(u)^2+4\dfrac{\p^2\pi^\perp\bg_1}{\p
b^2}(z_u)\ga_1(u)\\
&\odot\ga_3(u)+6\dfrac{\p^3\pi^\perp\bg_1}{\p
b^3}(z_u)\ga_1(u)^2\odot\ga_2(u)+4\dfrac{\p\pi^\perp \bg_{2}}{\p
b}(z_u)\ga_3(u)\\
&+4\dfrac{\p^3\pi^\perp\bg_1}{\p
b^3}(z_u)\ga_1(u)^3+12\dfrac{\p\pi^\perp \bg_{3}}{\p
b}(z_u)\ga_2(u)+12\dfrac{\p^2\pi^\perp \bg_{2}}{\p
b^2}(z_u)\ga_1(u)\\
 &\odot\ga_2(u)+12\dfrac{\p^2\pi^\perp \bg_{3}}{\p
b^2}(z_u)\ga_1(u)^2+24\dfrac{\p\pi^\perp \bg_{4}}{\p
b}(z_u)\ga_1(u)\Bigg)\vspace{0.3cm},\\
f_4(u)=&\dfrac{1}{24}\G_u \ga_4(u)+\dfrac{1}{24}\dfrac{\p^4\pi
\bg_1}{\p b^4}(z_u)\ga_1(u)^4+\dfrac{1}{4}\dfrac{\p^3\pi\bg_1}{\p
b^3}(z_u)\ga_1(u)^2\odot\ga_2(u)\vspace{0.3cm}\\
&+\dfrac{1}{8}\dfrac{\p^2\pi\bg_1}{\p
b^2}(z_u)\ga_2(u)^2+\dfrac{1}{6}\dfrac{\p^2\pi\bg_1}{\p
b^2}(z_u)\ga_1(u)\odot\ga_3(u)\vspace{0.3cm}\\
&+\dfrac{1}{6}\dfrac{\p^3\pi \bg_{2}}{\p
b^3}(z_u)\ga_1(u)^3+\dfrac{1}{2}\dfrac{\p^2\pi \bg_{2}}{\p
b^2}(z_u)\ga_1(u)\odot\ga_2(u)+\dfrac{1}{6}\dfrac{\p\pi
\bg_{2}}{\p b}(z_u)\ga_3(u)\vspace{0.3cm}\\
&+\dfrac{1}{2}\dfrac{\p^2\pi \bg_{3}}{\p
b^2}(z_u)\ga_1(u)^2+\dfrac{1}{2}\dfrac{\p\pi \bg_{3}}{\p
b}(z_u)\ga_2(u)+\dfrac{\p\pi \bg_{4}}{\p b}(z_u)\ga_1(u)+\pi \bg_{5}(z_u).
\end{array}
\end{equation}

The next result provides sufficient conditions for the existence of periodic solutions bifurcating from the set of non-isolated zeros $\CZ$ of the first-order averaged function.

\begin{theorem}[\cite{CLN}]\label{te}
Suppose that  $\bg_1(\bz_u)=0$ and $\det(\Delta_u )\neq 0$ for all $u \in \ov V$. 
 In addition, assume that, for some $m\in\{1,\ldots,k-1\}$, $f_0= \cdots f_{m-1}=0$ and $f_m\neq 0$.
 If there exists $u^*\in V$ such that $f_m(u^*)=0$ and $\det\left( Df_m(u^*)\right)\neq 0$, then
for $|\e|\neq 0$ sufficiently small there exists an isolated $T$-periodic solution $\f(t,\e)$ of system \eqref{tm13} such that $\f(0,0)=\bz_{u^*}$
\end{theorem}

\begin{remark}\label{rem1} 
As noticed above, the method developed in \cite{CLN} merges averaging theory and Lyapunov-Schmidt reduction in order to obtain sufficient conditions for the existence of an initial condition $\bz(\e)$ such that  $\f(t,\e)=\bx(t,\bz(\e),\e)$ is  an isolated $T$-periodic solution of \eqref{tm13}. 
This initial condition 
is written as  $\bz(\e)=\left( u(\e),\ov{ \mathcal{B}}(u(\e),\e)\right)$, where the function $u(\e)$ satisfies $\CF(u(\e), \e)=0$ for $|\e|\neq0$ sufficiently small.  Consequently, it is possible to write the expansion of $\bz(\e)$ around $\e=0$ using the bifurcation functions defined above. For instance, considering $m=2$ in Theorem $\ref{te}$, we  can write
\begin{equation}\label{zk}
\bz(\e)=\bz^*+\e \bz_{1}+\e^2\bz_{2}+\CO(\e^3).
\end{equation}
Denoting $\bz^*=(u^*,\mathcal{B}(u^*))$, $\bz_1=(u_1,\mathcal{B}_1)$, and $\bz_2=(u_2,\mathcal{B}_2)$, we have 
\[
\begin{array}{rl}
u_{1}=&-\left(D f_2(u^*)\right)^{-1}f_3(u^*),\vspace{0.2cm}\\ 
u_{2}=&-\dfrac{1}{2}\left(Df_2(u^*)\right)^{-1}\left(D^2f_2(u^*)u_1 \odot u_1+2Df_3(u^*)u_1+2f_4(u^*)\right),\vspace{0.2cm}\\
\mathcal{B}_{1}=&D\mathcal{B}(u^*)u_1+\gamma_1(u^*),\quad \text{and}\vspace{0.2cm}\\
\mathcal{B}_{2}=&\dfrac{1}{2}D^2\mathcal{B}(u^*)u_1\odot u_1+D\mathcal{B}(u^*)u_2+D\gamma_1(u^*)u_1+\gamma_2(u^*). 
 \end{array}
 \]
 The expression \eqref{zk}  will be used in Section $\ref{sec:stab}$ for determining the stability of the periodic solution $\bx(t,\bz(\e),\e)$.
\end{remark}

\subsection{Existence of Periodic Solutions - Case {\bf A}}\label{sec:psca}

The proof concerning the existence of a periodic solution of the R\"{o}ssler System \eqref{sr} bifurcating from the zero-Hopf equilibrium at the origin was provided in \cite{Ll} (see Theorem $2$ of \cite{Ll}). Here, for the sake of completeness, we again perform the proof of such a result using the first order averaging theory.

\begin{proposition}\label{p1}
Let  $(a, b, c)=(\ov{a}+\e \al_1+\e^2 \al_2, 1+\e \beta_1+\e^2\beta_2, \ov{a}+\e\gamma_1+\e^2\gamma_2)+\CO(\e^3)$, with $\ov{a}\in (-\sqrt{2}, \sqrt{2})\setminus \{0\}$, and $\e,\al_i,\beta_i,\gamma_i\in\R$ for $i=1,2$, and consider $d_0$ as defined in \eqref{CA}. 
If $d_0>0$,
then for $|\e|\neq 0$ sufficiently small the R\"{o}ssler System \eqref{sr} has a periodic solution $\varphi(t,\e)$ satisfying $\varphi(t,\e)\to (0,0,0)$ when $\e\to 0$.
\end{proposition}
\begin{proof}
As noticed in the introduction, for $(a, b, c)=(\ov{a}, 1, \ov{a}),$ with $\ov{a}\in (-\sqrt{2}, \sqrt{2})\setminus \{0\}$, the R\"{o}ssler System \eqref{sr} has a zero-Hopf equilibrium at the origin. As usual, this allow us to write the R\"{o}ssler System \eqref{sr} in the standard form \eqref{tm13} in order to use the averaging theory for detecting its periodic solutions. We start by writing the linear part of the R\"ossler System \eqref{sr} in its Jordan normal form, so consider the linear change of variables
$$(x,y,z)=\left(\ov{a} \left(\ov{Z}-\dfrac{\ov{Y}}{\sqrt{2-\ov{a}^2}}\right)+\ov{X},\dfrac{\ov{Y}}{\sqrt{2-\ov{a}^2}}-\ov{Z},-\dfrac{\left(\ov{a}^2-1\right) \ov{Y}}{\sqrt{2-\ov{a}^2}}+\ov{a} \ov{X}+\ov{Z}\right)
.$$
In addition, taking $(\ov X,\ov Y,\ov Z)=\e( X,  Y,  Z)$, we get
\begin{align*}
\dot{X}=&-\sqrt{2-\ov{a}^2} Y+\e \left(\dfrac{\al_1 \ov{a} Y}{\sqrt{2-\ov{a}^2}}-\al_1 \ov{a} Z\right)+\CO(\e^2),\\
\dot{Y}=&\sqrt{2-\ov{a}^2} X+\e\Bigg(\dfrac{Y \left(-\al_1+\left(2 \ov{a}^2-1\right) X+\ov{a} (\beta_1+\ov{a} (\al_1+\ov{a} Z-\gamma_1))+\gamma_1\right)}{\ov{a}^2-2}\\
&+\dfrac{\sqrt{2-\ov{a}^2}}{\left(\ov{a}^2-2\right)^2} \Big(-\ov{a}^4 Z (\al_1+X)+\ov{a}^3 \left(-\left(X^2-\gamma_1 X-Y^2+Z (\beta_1+Z)\right)\right)\\
&+\ov{a}^2 (X (Z-\beta_1)+Z (3 \al_1+\gamma_1))+\ov{a} \left(2 X (X-\gamma_1)-Y^2+2 Z (\beta_1+Z)\right)\\
&+2 X (\beta_1+Z)-2 Z (\al_1+\gamma_1)\Big)\Bigg)+\CO(\e^2),\\
\dot{Z}=&\e\Bigg(\dfrac{1}{\left(\ov{a}^2-2\right)^2}\Big(-\ov{a}^4 X Z-\ov{a}^3 \left(\left(X^2-\gamma_1 X-Y^2+Z (\beta_1+Z)\right)\right)\\
&+\ov{a}^2 (X (Z-\beta_1)+Z (\gamma_1-\al_1))+\ov{a} \left(2 X (X-\gamma_1)-Y^2+2 Z (\beta_1+Z)\right)\\
&+2 (X (\beta_1+Z)+Z (\al_1-\gamma_1))\Big)-\dfrac{\sqrt{2-\ov{a}^2} Y}{\left(\ov{a}^2-2\right)^2} \Big(\al_1+\left(2 \ov{a}^2-1\right) X\\
&+\ov{a} \left(\beta_1+\ov{a}^2 Z-\ov{a} \gamma_1\right)+\gamma_1\Big)\Bigg)+\CO(\e^2).
\end{align*}
Now, consider the cylindrical variables 
\begin{equation}\label{cylin}
(X,Y,Z)=(r\cos \theta,r\sin \theta, z).
\end{equation}
Notice that $\dot \T=\sqrt{2-\ov{a}^2}+\CO(\e),$ which is positive for $|\e|$ sufficiently small. Therefore, by doing a time-rescaling, $\theta$ can be taken as the new time of the system so that the R\"ossler System  \eqref{sr} becomes the following non-autonomous differential system
\begin{equation}\label{CAs}
\begin{split}
\dfrac{d r}{d \theta}=& \dfrac{\e \, r }{4 \left(\ov{a}^2-2\right)^2}\left(\ov{a} \left(3-2 \ov{a}^2\right) r \sin (3 \theta )-2 \left(\ov{a}^2-2\right) \sin (2 \theta ) \left(\beta_1+\ov{a}^2 z+\al_1 \ov{a}-\ov{a} \gamma_1+z\right)\right)\\
-&\sin \theta  \big(4 \al_1 \ov{a}^4 z+\ov{a}^3 \left(4 z (\beta_1+z)-2 r^2\right)-4 \ov{a}^2 z (3 \al_1+\gamma_1)+\ov{a} \left(r^2-8 z (\beta_1+z)\right)\\
+&8 z (\al_1+\gamma_1)\big)+\dfrac{\e\,\sqrt{2-\ov{a}^2}}{4 \left(\ov{a}^2-2\right)^2} \left(\cos \theta  \left(\left(1-2 \ov{a}^2\right) r^2+4 \al_1 \ov{a} \left(\ov{a}^2-2\right) z\right)\right.\\
+&\left.r \left(\left(2 \ov{a}^2-1\right) r \cos (3 \theta )-4 \sin ^2(\theta ) (-\al_1+\ov{a} (\beta_1+\ov{a} (\al_1+\ov{a} z-\gamma_1))+\gamma_1)\right)\right)+\CO(\e^2)\vspace{0.1cm}\\
=&\e\,F^1_1(\theta,r,z)+\e^2\,F^1_2(\theta,r,z)+\CO(\e^3),\vspace{0.3cm}\\
\dfrac{d z}{d \theta}=&-\dfrac{\e \,r \sin \theta }{\left(\ov{a}^2-2\right)^2} \left(\al_1+\left(2 \ov{a}^2-1\right) r \cos \theta +\ov{a} \left(\beta_1+\ov{a}^2 z-\ov{a} \gamma_1\right)+\gamma_1\right)\\
+& \dfrac{\e \sqrt{2-\ov{a}^2}}{2 \left(\ov{a}^2-2\right)^3}\left(2 \ov{a}^3 z (\beta_1+z)+\left(2 \ov{a}^2-3\right) \ov{a} r^2 \cos (2 \theta )+2 \left(\ov{a}^2-2\right) r \cos \theta  \big(\beta_1+\ov{a}^2 z\right.\\
-&\left.\ov{a} \gamma_1+z\big)+2 \ov{a}^2 z (\al_1-\gamma_1)-\ov{a} \left(r^2+4 z (\beta_1+z)\right)+4 z (\gamma_1-\al_1)\right)+\CO(\e^2)\vspace{0.1cm}\\
=&\e\,F^2_1(\theta,r,z)+\e^2\,F^2_2(\theta,r,z)+\CO(\e^3).
\end{split}
\end{equation}
Notice that the non-autonomous differential system \eqref{CAs} is  written in the standard form \eqref{tm13} for applying the averaging theorem. Thus, identifying
\[
t=\theta, \ T=2\pi, \ \bz=(r,z), \text{ and } \bF_1(\theta,r,z)=\Big(F^1_1(\theta,r,z),F_1^2(\theta,r,z)\Big),
\]
we compute the first-order averaged function \eqref{avf}, $\bg_1(r,z)=\big(\bg_1^1(r,z),\bg_1^2(r,z)\big)$, as
\begin{equation}\label{avfCA}
\begin{array}{rll}
\bg_1^1(r,z)=& \displaystyle\int^{2\pi}_0F^1_1(\theta,r,z)d\theta=&-\dfrac{r (-\al_1+\ov{a} (\beta_1+\ov{a} (\al_1+\ov{a} z-\gamma_1))+\gamma_1)}{2 \left(2-\ov{a}^2\right)^{3/2}},\vspace{0.2cm}\\

\bg_1^2(r,z)=&\displaystyle\int^{2\pi}_0F^2_1(\theta,r,z)d\theta
=& \Big(2 \ov{a}^2 z (\gamma_1-\al_1)-2 \ov{a}^3 z (\beta_1+z)+\ov{a} \left(r^2+4 z (\beta_1+z)\right)\\
&&+4 z (\al_1-\gamma_1)\Big)\dfrac{1}{2 \left(2-\ov{a}^2\right)^{5/2}}.
\end{array}
\end{equation}

The non-linear system $\bg_1(r,z) =(0,0)$ has two solutions $(\ov{r}_\pm,\ov{z})$, namely
\[
\ov{r}_\pm=\pm\dfrac{\sqrt{2 \left(2-\ov{a}^2\right)d_0}}{\ov{a}^3}\quad\text{and}\quad \ov{z}=\dfrac{\al_1-\al_1 \ov{a}^2+\left(\ov{a}^2-1\right) \gamma_1-\beta_1 \ov{a}}{\ov{a}^3}.
\]

Since the domain of the averaged function is $\R_+\times\R$, then for  $0<\ov{a}<\sqrt{2}$  only  the solution $(\ov{r}_+,\ov{z})$ is contained within the domain of $\bg_1$, and for $-\sqrt{2}<\ov{a}<0$ the only solution in the domain is $(\ov{r}_-,\ov{z})$. These solutions are the same as the ones obtained in \cite{Ll}. Moreover, the Jacobian determinant of $\bg_1$ at $(\ov{r}_\pm,\ov{z})$ is given by
\[
\begin{split}
\det\left(D\bg_1(\ov{r}_\pm,\ov{z})\right)=\dfrac{d_0}{\ov{a}^2 \left(\ov{a}^2-2\right)^3},
\end{split}
\]
and from hypothesis we have $\det\left(D\bg_1(\ov{r}_\pm,\ov{z})\right)\neq 0$. Thus, the result follows by applying Theorem  \ref{CAT} and going back through the change of variables \eqref{cylin}.
\end{proof}

\subsection{Existence of Periodic Solutions - Case {\bf B}}\label{sec:pscb}

Here, we are assuming that $(a,b,c)=(\al(\e), \omega^2-1+\beta(\e), \gamma(\e))$,  where $\omega>0$, $\omega \notin\{ 1, \sqrt{2}\}$,
\[
\al(\e)=\sum_{i=1}^5\e^i\al_i+\CO(\e^6), \,\, \beta(\e)=\sum_{i=1}^5\e^i\beta_i+\CO(\e^6),\,\, \text{and}\,\, \gamma(\e)=\sum_{i=1}^5\e^i\gamma_i+\CO(\e^6).
\]
\begin{proposition}\label{p2}
Consider $\de$ as defined in \eqref{CB} and assume that $\al_1=\gamma_1(\omega^2-1)$, $\al_2=\beta_1 \gamma_1+\gamma_2 (\omega^2-1)$, $\omega>0$, $\omega\not\in \{1,\sqrt{2}\}$ and $\delta>0$. Then,  for  $|\e|\neq0$ sufficiently small, the R\"{o}ssler System \eqref{sr} has a periodic solution $\varphi(t,\e)$ satisfying $\varphi(t,\e)\to (0,0,0)$ when $\e\to 0$.
\end{proposition}
\begin{proof}
As noticed in the introduction, for $(a, b, c)=(0, \ov{b}, 0),$ with $\ov{b}\in (-1, +\infty)$, the R\"{o}ssler System \eqref{sr} has a zero-Hopf equilibrium at the origin. Again, this allow us to write the R\"{o}ssler System \eqref{sr} in the standard form \eqref{tm13} in order to use the averaging theory for detecting its periodic solutions. We start by writing the linear part of the R\"ossler System  \eqref{sr} in its Jordan normal form, so consider the linear change of variables
$$(x,y,z)=\left(\ov X,\dfrac{\ov Y}{\omega }+\ov Z,\ov Y \left(\omega -\dfrac{1}{\omega }\right)-\ov Z \right).$$
In addition, taking $(\ov  X, \ov  Y, \ov Z)=\e(X,Y,Z)$, we see that the unperturbed system (that is, $\e=0$) in these new variables can be written as $\big(\dot X,\dot Y,\dot Z\big)=\big(- \omega Y,\omega X,0\big)$. Thus, considering cylindrical coordinates $(X,Y,Z)=(r\cos\theta,r\sin \theta,z)$, we see that $\dot \T=\omega+\CO(\e),$ which is positive for $|\e|$ sufficiently small. Therefore, by doing a time-rescaling, $\theta$ can be taken as the new time of the system so that the R\"ossler System \eqref{sr} becomes the following non-autonomous differential system
\begin{equation}\label{s22}
\dfrac{d r}{d \theta}=\sum_{i=1}^5\e^i F_{i}^1(\theta,r,z)+\CO(\e^6),\quad
\dfrac{d z}{d\theta}=\sum_{i=1}^5\e^i F_{i}^2(\theta,r,z)+\CO(\e^6)
\end{equation}
 where $(\theta,r,z)\in\R\times\R_+\times \R$. Due to the extent of the expressions of $F_i^j(\theta,r,z)$, $i=1,\ldots,5$ and $j=1,2$, we shall omit them here. However, they are trivially computed in terms of the parameters $\omega$, $\al_i,\beta_i,\gamma_i$, $i=1,\ldots,5.$

Notice that the non-autonomous differential system \eqref{s22} is  written in the standard form \eqref{tm13} for applying the averaging theorem. Thus, identifying
\[
t=\theta, \ T=2\pi, \ \bz=(r,z),  \text{ and } \bF_{i}(\theta,r,z)=\left( F_{i}^1(\theta,r,z), F_{i}^2(\theta,r,z)\right) \text{ for } i=1,\ldots,5,
\]
we compute the first-order averaged function \eqref{avf},
\begin{equation}\label{CBg1}
\bg_1(r,z)=\left(\dfrac{\pi  r \left(\al_1+\gamma_1 (1-\omega ^2)\right)}{\omega ^3},-\dfrac{2 \pi  z \left(\gamma_1+\al_1 (1-\omega ^2)\right)}{\omega ^3}\right),
\end{equation}
for $(r,z)\in\R_+\times\R.$ This function only has the trivial zero, which is not contained within the domain and, therefore, does not correspond to a periodic solution of \eqref{s22}. Consequently, no periodic solutions can be detected using the first-order averaged function. This fact had already been noticed in \cite{Ll}.

In order to follow the averaging method, we compute the second-order averaged function \eqref{avf}, $\bg_2(r,z)=\big(g^1_2(r,z),g^2_2(r,z)\big)$, as
\begin{align*}
g_2^1(r,z)=&\dfrac{\pi }{2 \omega ^6} \Big(\pi  r \left(\al_1+\gamma_1 \left(1-\omega ^2\right)\right)^2-\omega ^3 (r (\al_1 z-2 (\al_2+\gamma_2)+\gamma_1 (3 z-\beta_1))\\
&-2 \gamma_2 r \omega ^5\hspace{-0.8 mm}+2 z (\al_1\hspace{-0.8 mm}+\hspace{-0.8 mm}\gamma_1) (2 \al_1\hspace{-0.8 mm}+\hspace{-0.8 mm}\gamma_1))+\omega  (\al_1+\gamma_1) (r (4 z-3 \beta_1)+6 z (\al_1+\gamma_1))\hspace{-0.8 mm}\Big),\\
g_2^2(r,z)=&\dfrac{\pi }{2 \omega ^7} \Big(r^2 \left(1-\omega ^2\right) \left(\al_1 \left(\omega ^2-1\right)+\gamma_1 \left(2 \omega ^2-1\right)\right)+2 r \left(\omega ^2-1\right) (\al_1+\gamma_1) \\
&\left(\al_1 \left(2 \omega ^2-3\right)+\gamma_1 \left(\omega ^2-3\right)\right)+2 \omega  z \Big(2 \pi  \left(\al_1 \left(1-\omega ^2\right)+\gamma_1\right)^2+2 \al_2 \omega ^5\\
&+\omega ^3 (\al_1 (z-\beta_1)-2 (\al_2+\gamma_2))-3 \omega  (\al_1+\gamma_1) (z-\beta_1)\Big)\Big).
\end{align*}

In the research literature, the next natural step would usually consist in assuming some constraints on the first-order parameters (that is, $\alpha_1,$ $\beta_1,$ and $\gamma_1$) such that the first-order averaged function vanishes identically, and then computing the zeros of the second-order averaged function. This method can be implemented at any order of perturbation. Nevertheless, we shall see in Section \ref{sec:SA} that this procedure fails in providing periodic solutions up to order 5. Here, instead of vanishing the first-order averaged function, we shall use Theorem \ref{te} as follows.

Firstly, take $\al_1=\gamma_1(\omega^2-1)$. Thus, the first-order averaged function becomes 
$$
\bg_1(r,z)=\left( 0, \dfrac{2 \pi  \gamma_1 \left(\omega ^2-2\right) z}{\omega }\right).
$$
Notice that, in this case, the first-order averaged function has a continuum of zeros
$\CZ=\left\lbrace
\bz_r=\left(r,0\right):r>0
\right\rbrace.$ Moreover, the Jacobian matrix of $\bg_1$ on $\CZ$ can be written as 
$$
D{\bg_1}(\bz_r)=\left(
\begin{array}{cc}
 0 & 0 \\
 0 & \dfrac{2 \pi  \gamma_1 \left(\omega ^2-2\right)}{\omega } \\
\end{array}
\right).
$$

Thus, we  compute the first-order bifurcation function \eqref{bif1} as
\begin{align*}
f_1(r)=&\dfrac{\pi  r \left(\al_2-\beta_1 \gamma_1+\gamma_2 \left(1-\omega ^2\right)\right)}{\omega ^3}.
\end{align*}
This function has no positive simple zeros. In order to use Theorem \ref{te}, we  should impose some constrains on the parameters appearing in $f_1$ (that is $\beta_1,$ $\gamma_1,$ $\alpha_2,$ and $\gamma_2$) in order to vanish $f_1,$ then computing the zeros of the second-order bifurcation function $f_2.$ For that, we take $\al_2=\beta_1 \gamma_1+\gamma_2 \left(\omega ^2-1\right)$.

Now, in order to obtain the second-order bifurcation function \eqref{bif2} we must compute the third-order averaged function \eqref{avf}, $\bg_3(r,z)=\big(g^1_3(r,z),g^2_3(r,z)\big)$, as
\begin{align*}
g_3^1(r,z)=&\dfrac{\pi}{16 \omega ^5}  \Big(\gamma_1r^3 \left(1- \omega ^2\right)+16 \gamma_1^2 r^2 \left(\omega ^2-1\right)+4 r 
\Big(4 \omega ^2 \Big(\al_3-\beta_1 \gamma_2-\beta_2 \gamma_1\\
&-3 \gamma_1^3 \left(\omega ^4-3 \omega ^2+2\right)-\gamma_3 \omega ^2+\gamma_3\Big)+\gamma_1 \left(8-3 \omega ^2\right) z^2+z \left(\beta_1 \gamma_1 \left(\omega ^2-6\right)\right. \\
&-\left.2 \omega  \left(\omega ^2-2\right) \left(\pi  \gamma_1^2 \left(\omega ^2-2\right)+\gamma_2 \omega \right)\right)\Big)-8 \gamma_1 \omega ^2 z \left(2 \left(\beta_1 \gamma_1 \left(\omega ^2+2\right)+2 \omega\right.\right.\\
&\left. \left. \left(\omega ^2-2\right) \left(\pi  \gamma_1^2 \left(\omega ^2-2\right)+2 \gamma_2 \omega \right)\right)+\gamma_1 \left(11 \omega ^2-26\right) z\right)\Big) ,\\
g_3^2(r,z)=& \dfrac{\pi}{24 \omega ^7}  \Big(r^2 \left(6 \omega ^2 \left(\beta_1 \gamma_1 \left(\omega ^2-3\right)-2 \omega  \left(\omega ^2-1\right) \left(\pi  \gamma_1^2 \left(\omega ^2-2\right)+\gamma_2 \omega \right)\right)\right.\\
&+\left.\left.\gamma_1 \left(2 \omega ^6-29 \omega ^4+37 \omega ^2-10\right) z\right)+12 \gamma_1 r \omega ^2 \left(2 \left(\beta_1 \gamma_1 \left(\omega ^4+3 \omega ^2-6\right)\right.\right. \right. \\
&+\left.\left. \left.  2 \omega  \left(\omega ^4-3 \omega ^2+2\right) \left(\pi  \gamma_1^2 \left(\omega ^2-2\right)+2 \gamma_2 \omega \right)\right)+3 \gamma_1 \left(\omega ^4-7 \omega ^2+6\right) z\right)\right.\\
&+ 2 \omega ^2 z \Big(4 \omega ^2 \Big(6 \al_3 \left(\omega ^2-1\right)-3 \beta_2 \gamma_1 \left(\omega ^2-4\right)+\gamma_1 \left(\omega ^2-2\right) \Big(\gamma_1^2 \Big(15 \left(\omega ^2-1\right) \\
&+4 \pi ^2 \left(\omega ^2-2\right)^2\Big)+12 \pi  \gamma_2 \omega  \left(\omega ^2-2\right)\Big)-6 \gamma_3\Big)-3 \beta_1^2 \gamma_1 \left(\omega ^2+6\right)+12 \beta_1 \omega  \\
&\left(2 \pi  \gamma_1^2 \left(\omega ^4-4\right)-\gamma_2 \omega  \left(\omega ^2-4\right)\right)+9 \gamma_1 \left(\omega ^2-6\right) z^2+6 z \Big(2 \omega  \left(\omega ^2-4\right)  \\
&\left(3 \pi  \gamma_1^2 \left(\omega ^2-2\right)+\gamma_2 \omega \right)-\beta_1 \gamma_1 \left(\omega ^2-12\right)\Big)\Big)\Big).
\end{align*}
Then, the second-order bifurcation function \eqref{bif2} can be written as

\begin{equation}\label{f2CB}
 f_2(r)=\dfrac{3 \pi  \gamma_1  \left(\omega ^2-1\right)r}{16 \omega ^5}\left(r^2-\dfrac{16 \omega ^2}{3} \delta\right),
\end{equation}
where 
$$
\delta=\dfrac{\beta_1 \gamma_2+\beta_2 \gamma_1+\gamma_1^3 \left(\omega ^4-3 \omega ^2+2\right)+\gamma_3 \left(\omega ^2-1\right)-\alpha_3}{\gamma_1 \left(1-\omega ^2\right)}.$$
Thus, for $\delta>0$, the second-order bifurcation function \eqref{f2CB} have the following unique simple zero within the domain $\R_+$,
$$
r^*=4\omega\sqrt{\dfrac{\delta}{3}}.
$$ 
The result follows directly from Theorem \ref{te} with $m=2.$
\end{proof}

\subsection{Fifth-Order Standard Analysis}\label{sec:SA}
In this section, we shall apply the usual higher order averaging method up to order five for studying periodic solutions of the non-autonomous differential system \eqref{s22}. We shall see that, in this case, this method does not provide any information about the existence of periodic solutions, which emphasizes the importance of the method developed in \cite{CLN} and employed in the proof of Proposition \ref{p2}.

Consider the first-order averaging function $\bg_1$ \eqref{CBg1}. As noticed in the proof of Proposition \ref{p2}, the non-linear system $\bg_1(r,z)=(0,0)$ has no solution in the domain $\R_+\times\R$. Therefore, as said before, in order to use Theorem \ref{CAT} for detecting periodic solutions of \eqref{s22}, we could assume values for the first-order parameters perturbation, $\al_1,\beta_1$, and $\gamma_1$, such that $\bg_1= 0$, and then computing the zeros of the second-order averaging function $\bg_2$. This procedure can be implemented at any order and is the usual way of applying the higher order averaging method for studying periodic solutions.

Notice that, for $\omega\neq \sqrt{2}$, $\bg_1= 0$ if, and only if, $\al_1=\gamma_1=0$. By assuming these values, the second-order averaged function can be written as
\begin{align*}
\bg_2(r,z)=\left(\dfrac{\pi  r \left(\al_2+\gamma_2 (1-\omega ^2)\right)}{\omega ^3},-\dfrac{2 \pi  z \left(\gamma_2+\al_2 (1-\omega ^2)\right)}{\omega ^3}\right),
\end{align*}
which is the same expression of $\bg_1$, just replacing $\al_1$ and $\gamma_1$ by $\al_2$ and $\gamma_2$, respectively. As before, the non-linear system $\bg_2(r,z)=(0,0)$ has no solution in the domain $\R_+\times\R$ and $\bg_2= 0$ if, and only if, $\al_2=\gamma_2=0$.

For $l=1,\ldots,4$, we can check that $\al_1=\ldots\al_l=\gamma_1=\ldots\gamma_l=0$ implies that
\begin{align*}
\bg_{l+1}(r,z)=\left(\dfrac{\pi  r \left(\al_{l+1}+\gamma_{l+1} (1-\omega ^2)\right)}{\omega ^3},-\dfrac{2 \pi  z \left(\gamma_{l+1}+\al_{l+1} (1-\omega ^2)\right)}{\omega ^3}\right).
\end{align*}
Again, the non-linear system $\bg_{l+1}(r,z)=(0,0)$ has no solution in the domain $\R_+\times\R$ and $\bg_{l+1}= 0$ if, and only if, $\al_{l+1}=\gamma_{l+1}=0$.

Consequently, up to order five, the usual recursive method does not provide any information about the existence of periodic solutions of the differential system \eqref{s22}.

\section{Stability of Periodic Solutions}\label{sec:stab}

In Section \ref{sec:ps}, sufficient conditions for the existence of periodic solutions for the R\"ossler System \eqref{sr} were provided. In this section, the stability of such periodic solutions will be studied. We shall essentially demonstrate that the periodic solution provided in {\bf Case A} has its stability determined by the Jacobian matrix of the first-order averaged function, which in this case is hyperbolic. On the other hand, the periodic solution provided in {\bf Case B} does not have its stability determined by the Jacobian matrix of the first-order averaged function, which in this case is not hyperbolic, so that studying its stability demands a more refined analysis.

As before, let  $\bx(t,\bz,\e)$ denote the solution of \eqref{tm13} satisfying $\bx(0,\bz,\e)=\bz$. As commented in Section \ref{sec:at}, the essence of Theorems \ref{CAT} and \ref{te} is to provide sufficient conditions that guarantee the existence of an initial condition $\bz(\e)\in \Omega$ such that $\f(t,\e)=\bx(t,\bz(\e),\e)$ is a branch of isolated $T$-periodic solutions of system \eqref{tm13}. From \eqref{quad},  the Poincar\'{e} map $\Pi(\bz,\e)=\bx(T,\bz,\e)$  of system \eqref{tm13} is given by
\begin{equation*}
\Pi(\bz,\e)=\bz+\e \bg_1(\bz)+\e^2\bg_2(\bz)+\e^3 \bg_3(\bz)+\CO(\e^4),
\end{equation*}
where  $\bg_i$ is the averaged function of order $i$ defined in \eqref{avf}. 
The stability of the periodic solution $\f(t,\e)$ can be determined by the eigenvalues of the Jacobian matrix $D_{\bz}\Pi(\bz(\e),\e)$, which can be expanded around $\e=0$ as follows:
\begin{equation*}\label{gk}
D_{\bz}\Pi(\bz(\e),\e)=Id+\e A_0+\e^2 A_1+\e^3 A_2+\CO(\e^3),
\end{equation*}
where $A_0= D{\bg_1}(\bz^*),$ $\bz^*=\bz(0)$. Recall that, if its eigenvalues of $D_{\bz}\Pi(\bz(\e),\e)$, $\widetilde\lambda_1(\e)$ and $\widetilde\lambda_2(\e),$  satisfy $|\widetilde\lambda_1(\e)|<1$ and $|\widetilde\lambda_2(\e)|<1$, then the  periodic solution $\f(x,\e)$  is asymptotically stable. Otherwise, it is unstable.

The next result provides the stability of the periodic solution $\f(t,\e)$ when the matrix $A_0$  is hyperbolic, that is, it has no eigenvalues over the imaginary axis of the complex plane:
\begin{theorem}[\cite{verh}]\label{teostab}
Consider the differential system \eqref{tm13} and suppose that the conditions of Theorem \ref{CAT} are satisfied for $m=1$. If all eigenvalues of  $D{\bg_1}(\bz^*)$  have negative real parts, then the corresponding periodic solution $\varphi(t, \e)$ of system \eqref{tm13} is asymptotically stable for  $\e>0$ sufficiently small. Conversely, if one of the eigenvalues has positive real part, then $\varphi(t,\e)$ is unstable.
\end{theorem} 
This theorem will be used for studying the stability of the periodic solutions detected in Proposition \eqref{p3}.
 
\subsection{$k$-Determined Hyperbolic Matrices} If $A_0$ is not hyperbolic, then the former theorem cannot be used  to analyze  the stability of the periodic solution $\f(t,\e)$. In this case,  we shall need the next result about $k-$determined hyperbolic matrices (see \cite[Chapter $3$]{Mu}). Roughly speaking, we say that a smooth matrix $A(\e)$, defined in a neighborhood of $\e=0$, is $k-$hyperbolic when the hyperbolicity of $A(\e)$ is determined by the hyperbolicity of its $k-$jet (see \cite{Mu}).

\begin{theorem}[{{\cite[Theorem~$3.7.7$]{Mu}}}]\label{mk}
Suppose that $C(\e)$ and $D(\e)$ are continuous matrix-valued functions defined for $\e>0$ and that
\begin{equation}\label{ce}
C(\e)=\Lambda(\e)+\e^R D(\e),
\end{equation}
where
$$
\Lambda(\e)=\left[\begin{matrix}
\lambda_1(\e)& & &\\
&\ddots&& \\
& &\ddots& \\
&&&\lambda_n(\e)
\end{matrix} \right]=\e^{r_1} \Lambda_1+\dots+\e^{r_j}\Lambda_j.
$$
Here, $r_1<r_2<\dots < r_j<R$  are rational numbers, and $\Lambda_1,
\dots, \Lambda_j$ are diagonal matrices. Then, there exists $\e_0>0$
such that for $0<\e<\e_0$ the eigenvalues of $C(\e)$ are approximately equal to the diagonal
entries $\lambda_i(\e)$ of $\Lambda(\e)$, with error $\CO(\e^R)$.
\end{theorem}

The theorem above will be applied as follows. Assume that $A(\e)$ is a smooth matrix function defined in a neighborhood of $\e=0$. Suppose that there exists an invertible matrix $T(\e)$, defined for $\e>0$ sufficiently small, such that the fractional power series of $T(\e)^{-1}A(\e)T(\e)$ can be written as \eqref{ce} and satisfies the hypotheses of Theorem \ref{mk}. Since the matrices $A(\e)$ and $T(\e)^{-1}A(\e)T(\e)$ are similar for  $\e>0$ sufficiently small, we conclude from Theorem \ref{mk} that the eigenvalues of $A(\e)$ are approximately equal to the diagonal entries $\lambda_i(\e)$ of $\Lambda(\e)$, with error $\CO(\e^R)$.

\subsection{Stability of Periodic Solutions - {\bf Case  A}}
In this case, we shall see that, under the hypotheses of Theorem \ref{t1},  the Jacobian matrix of the first-order averaged function is hyperbolic. As such, the next result follows in straightforward manner. 

\begin{proposition}\label{p3}
Consider $d_0>0$ and $d_1$ as defined in \eqref{CA}. The periodic solution provided by Proposition $\ref{p1}$ is asymptotically stable (resp. unstable) provided that $d_1> 0$ (resp.  $d_1<0$).
\end{proposition}

\begin{proof}
Consider the first-order averaging function $\bg_1$ defined in \eqref{avfCA}, of the non-autonomous differential system \eqref{CAs}. 
According to the proof of Proposition \ref{p1},  $\bz_{+}=(\ov r_{+},\ov z)$ (resp. $\bz_{-}=(\ov r_{-},\ov z)$) is the solution of the non-linear equation ${\bf g}_1(\bz)=0$, provided that $0<{\ov a}< \sqrt{2}$ (resp. $-\sqrt{2}<{\ov a}< 0$). Moreover, the Jacobian matrix of ${\bf g}_1(\bz)$ at $\bz_{\pm}$
has the following characteristic polynomial
$$
p(\lambda)=\lambda^2+\dfrac{d_1}{{\ov a}^2 \sqrt{2-{\ov a}^2}}  \lambda + \dfrac{d_0}{{\ov a}^2 \left(2-{\ov a}^2\right)^3}.
$$
According to Theorem \ref{teostab}, we know that the stability of the periodic solution concerning the differential system can be determined by the roots of $p(\lambda)$, provided that they are not within the imaginary axis. By  the Routh-Hurwitz test we have that the roots of $p(\lambda)$ will be in the left side of the complex plane if,  and only, if $d_1>0$ and $d_0>0$. On the other hand, if $d_0\,d_1< 0$ the polynomial $p(\lambda)$ will have at least one root with positive real part and, consequently, the periodic solution will have at least one unstable direction. 
\end{proof}

\subsection{Stability of Periodic Solutions - Case {\bf B}} In this case, we shall see that the Jacobian matrix of the first-order averaged function, $A_0=D{\bg_1}(\bz_{\al^*})$, is not hyperbolic. Thus, the theory of $k$-determined hyperbolic matrices will be employed in order to obtain the following result.

\begin{proposition}\label{p4} Considering $\la_1$ and $\la_2$ as defined in \eqref{CA}, the following statements hold:
\begin{itemize}
\item[$(a)$] If $\lambda_1<0$ and $\lambda_2<0$, then the periodic solution detected by Proposition \ref{p2} is asymptotically stable.
\item[$(b)$] If $\lambda_1 \lambda_2<0$, then the periodic solution detected by Proposition \ref{p2} is unstable. Moreover, it admits  stable and unstable manifolds, which are locally characterized by topological cylinders.
\item[$(c)$] If  $\lambda_{1}>0$ and  $\lambda_{2}>0$, then the periodic solution detected by Proposition \ref{p2} is unstable. Moreover, the unstable manifold has dimension 3.
\end{itemize}

\end{proposition}

\begin{proof}
According to the proof of Proposition \ref{p2}, we have that the second-order bifurcation function $f_2(r)$, defined in \eqref{f2CB}, has the simple zero $r^*$.  In accordance with Remark \ref{rem1}, this zero is related to an initial condition $\bz(\e)$ such that $\bx(t,\bz(\e),\e)$ is periodic. Moreover,
\begin{equation*}\label{zki}
\bz(\e)=\bz^*+\e \bz_{1}+\e^2\bz_{2}+\CO(\e^3),
\end{equation*}
where $\bz^*=(r^*,0)$, $\bz_1=\left(-\dfrac{f_3(r^*)}{f_2'(r^*)},\gamma_1(r^*)  \right)$, and
\begin{align*}
\bz_2=&\hspace{-0.8mm}\left(\hspace{-0.8mm}\dfrac{2f_3(r^*)f_2'(r^*)f_3'(r^*)\hspace{-0.8mm}-\hspace{-0.8mm}2f_4(r^*)(f_2'(r^*))^2\hspace{-0.8mm}-\hspace{-0.8mm}f_3^2(r^*)f_2''(r^*)}{2(f_2'(r^*))^{3}},\gamma_2(r^*)\hspace{-0.8mm}-\hspace{-0.8mm}\dfrac{f_3(r^*)}{f_2'(r^*)}\gamma_1'(r^*) \hspace{-0.8mm} \right)\hspace{-0.8mm}.
\end{align*}
We see in the expressions above that the bifurcation functions of orders three and four, $f_3$ and $f_4$, are needed. Additionally, in their definitions \eqref{bif3} and \eqref{bif4}, respectively, we see that the averaged functions of  orders four and five, $\bg_4$ and $\bg_5$, are also needed. Due to the extent of these expressions, we shall omit them here.

Now, we are able to compute the expansion of the Jacobian matrix $D_{\bz}\Pi(\bz(\e),\e)$ around $\e=0$ as  
\begin{equation}\label{partialPi}
D_{\bz}\Pi(\bz(\e),\e)=Id+\e A_0+\e^2 A_1+\e^3 A_2+\CO(\e^3),
\end{equation}
Denoting $A_j=(A_j(l,k))$, $j=0,1,2,$ we can easily see that
\begin{equation*}
A_0(1,1)=A_0(1,2)=A_0(2,1)=0,\quad\text{and}\quad
A_0(2,2)=\dfrac{2 \pi  \gamma_1 \left(\omega ^2-2\right)}{\omega}.
\end{equation*}
Due to the extent of the expressions of $A_j=(A_j(l,k))$, for $j=1,2$, we shall also omit them here. Notice that, they are computed in terms of the parameters $\omega$, $\al_i$, for $i=3,\ldots,5$, and $\beta_i,\gamma_i$, for  $i=1,\ldots,5.$  

Clearly, all eigenvalues of \eqref{partialPi} have the form $\widetilde\lambda_i(\e)=1+\e\lambda_i(\e)$, where $\lambda_i(\e)$ is an eigenvalue of  $A(\e)=A_0+\e A_1+\e^2 A_2+\CO(\e^3)$. In what follows, we apply Theorem \ref{mk} for studying the eigenvalues of the matrix $A(\e)$. First, define 
$$
T(\e)=\begin{pmatrix}
1 & T_{12}\\
T_{21} & 1
\end{pmatrix},
$$
with 
\[
\begin{array}{rl}
T_{12}=&-\dfrac{A_1(1,2)\e}{A_0(2,2)}+\dfrac{\left(A_1(1,2)A_1(2,2)-A_0(2,2)A_2(1,2)\right)\e^2}{(A_0(2,2))^2},\vspace{0.2cm}\\
T_{21}=&\dfrac{A_1(2,1)\e}{A_0(2,2)}+\dfrac{\left(-A_1(2,1)A_1(2,2)+A_0(2,2)A_2(2,1)\right)\e^2}{(A_0(2,2))^2}.
\end{array}
\]
Notice that the matrix $T(\e)$ is invertible for $|\e|\neq0$ sufficiently small. Moreover,
\begin{equation*}\label{s24}
T(\e) A(\e) T^{-1}(\e)=\Lambda_0+\e \Lambda_1+\e^2 \Lambda_2+\CO(\e^3),
\end{equation*}
where the matrices  $\Lambda_j=(\Lambda_j(lk))$, for $j=0,1,2$, are diagonal (that is, $\Lambda_j(12)=\Lambda_j(21)=0$) and satisfy
\begin{align*}
\Lambda_0(11)=& 0,\\
\Lambda_0(22)=&\gamma_1 \left(\omega ^2-2\right)\dfrac{2 \pi}{\omega},\\
\Lambda_1(11)=&0 ,\\
\Lambda_1(22)=&\dfrac{\pi  \left({\beta_1} {\gamma_1} \left(\omega ^2+2\right)+2 \omega  \left(\omega ^2-2\right) \left(\pi  {\gamma_1}^2 \left(\omega ^2-2\right)+{\gamma_2} \omega \right)\right)}{\omega ^3}\\
\Lambda_2(11)=& \dfrac{2 \pi\,  \delta\, \gamma_2(1-\omega^2)}{\omega ^3}  ,\\
\Lambda_2(22)=&\dfrac{\pi}{36 \omega ^5 \left(\omega ^2-2\right)}  \Big(8 {\alpha_3} \left(11 \omega ^6-40 \omega ^4+22 \omega ^2-20\right)-9 {\beta_1}^2 {\gamma_1} \left(\omega ^2-2\right) \left(\omega ^2+6\right)\\
+&72 \pi  {\beta_1} {\gamma_1}^2 \omega  \left(\omega ^2-2\right)^2 \left(\omega ^2+2\right)+4 {\beta_1} {\gamma_2} \left(-13 \omega ^6+80 \omega ^4-80 \omega ^2+40\right)\\
+&4 \left({\beta_1} {\gamma_1} \left(-13 \omega ^6+80 \omega ^4-80 \omega ^2+40\right)+{\gamma_1}^3 \left(\omega ^2-2\right) \left(\omega ^2 \left(-4 \omega ^6+57 \omega ^4 \phantom{\big)^3}\right.\right.\right.\\ 
-&\left.\left.\left.\left.115 \omega ^2+12 \pi ^2 \left(\omega ^2-2\right)^3+102\right)-40\right)+36 \pi  {\gamma_1} {\gamma_2} \omega ^3 \left(\omega ^2-2\right)^3\right.\right.\\
-&\left.2 {\gamma_3} \left(2 \omega ^8-15 \omega ^6+26 \omega ^4-42 \omega ^2+20\right)\right)\Big).
\end{align*}
Thus, matrix $T(\e) A(\e) T^{-1}(\e)$ has the form \eqref{ce}. Since the matrices $A(\e)$ and $T(\e) A(\e) T^{-1}(\e)$ are similar for $|\e|\neq0$ sufficiently small, it follows from Theorem \ref{mk} that the eigenvalues of $A(\e)$ are written as
\[
\lambda_1(\e)=\lambda_1\dfrac{2\pi}{\omega}+\CO(\e)\,\, \text{and}\,\, \lambda_2(\e)=\e^2 \lambda_{2}\dfrac{2\pi \delta}{\omega^3}+\CO(\e^3),
\]
where $\lambda_1=\dfrac{\omega}{2\pi}\Lambda_0(22)$, $\lambda_{2}=\dfrac{\omega^3}{2\pi \delta}\Lambda_2(11)$, $\omega>0,$ and $\delta>0$. 

Consequently, the eigenvalues of \eqref{partialPi} are written as
\[
\widetilde \lambda_1(\e)=1+\e \lambda_1\dfrac{2\pi}{\omega^3}+\CO(\e^2)\,\, \text{and} \,\,
\widetilde\lambda_2(\e)=1+\e^3 \lambda_{2}\dfrac{2\pi \delta}{\omega^3}+\CO(\e^4).
\]
 Thus, 
$$
|\widetilde\lambda_1(\e)|^2=1+\e\lambda_1\dfrac{4\pi}{\omega}+\CO(\e^2) \mbox{ and } |\widetilde\lambda_2(\e)|^2=1+\e^3 \lambda_{2}\dfrac{4\pi \delta}{\omega^3}+\CO(\e^4).
$$
Therefore, since $\delta>0,$ we have that for $\e>0$ sufficiently small,  $|\widetilde \lambda_1(\e)|\gtrless1$ and $|\widetilde \lambda_2(\e)|\gtrless1$ provided that $ \lambda_1\gtrless0$ and $ \lambda_{2}\gtrless0$, respectively. From here,  statements $(a)$, $(b)$, and $(c)$ follow in straightforward manner.
\end{proof}

\section{Bifurcation of an Invariant Torus}\label{sec:it}

In \cite{ITCanNov2018}, the following two-parameter family of non-autonomous differential systems was considered:
\begin{equation}\label{tm1}
\dot {\bf x}=\sum_{i=1}^k \e^i {\bf F}_i(t,{\bf
x},\mu)
+\e^{k+1} \widetilde{{\bf
F}}(t,{\bf x},\mu,\e), \quad  (t,{\bf x},\mu,\e)\in\R\times\Omega\times J \times(-\e_0,\e_0),
\end{equation}
where $\Omega$ is an open bounded subset of $\mathbb{R}^2,$ $J$ is a open interval, and $\e_0$ is a small positive real number. It is assumed that $\bF_i,$ $i=1,\ldots,k,$ and $\widetilde\bF$ are sufficiently smooth functions $T$-periodic in the variable $t.$ Again, the periodicity of system \eqref{tm1} allow us to see it defined in the cylinder $(t,\bx) \in \s^1\times \Omega,$ where $\s^1\equiv \R/T\mathbb{Z}.$ Notice that system \eqref{tm1} is written in the standard form \eqref{tm13} of the averaging theory with an additional  parameter $\mu$ distinguished.  It has been provided generic conditions on the averaged functions \eqref{avf} guaranteeing the existence of a codimension-one bifurcation curve $\mu(\varepsilon)$ in the parameter space $(\mu,\varepsilon)$ characterized by the birth of an invariant torus of \eqref{tm1} in $\s^1\times \Omega$ from a periodic solution. In this section, we first introduce the main result obtained in \cite{ITCanNov2018}, then we apply it to conclude the proof of Theorem \ref{t1}. 

Let $g_m(\bx,\mu)$ be the first non-vanishing averaged function of system \eqref{tm1}.  The strategy followed by \cite{ITCanNov2018} consisted in looking for conditions that ensure a {\it Neimark-Sacker Bifurcation} (see \cite{k}) in the Poincar\'{e} map of system \eqref{tm1},
\begin{equation}\label{quad2}
\Pi(\bz,\mu,\e)=\bx(T,\bz,\mu,\e)=\bz+\sum_{i=m}^k\e^i\bg_i(\bz,\mu) +\CO(\e^{k+1}),
\end{equation}
$\bz\in\Sigma=\{t=0\}.$ Again, $\bx(t,\bz,\mu,\e)$ denotes the solution of \eqref{tm1} satisfying $\bx(0,\bz,\mu,\e)=\bz.$ In discrete dynamical system theory, this bifurcation is characterized by the birth of an invariant closed curve from a fixed point, as the fixed point changes stability. As it well known,
an invariant torus corresponds to an invariant closed curve $\Gamma\subset\Sigma$ of $\Pi(\bz,\mu,\e)$, that is, $\Pi(\Gamma,\mu,\e)=\Gamma.$

As a first hypothesis, we assume that:
 \begin{itemize}
\item[{\bf H1.}] 
{\it there exists a continuous curve $\mu\in J\mapsto\bx_\mu \in \Omega,$ defined in an interval $J\ni \mu_0,$ such that $\bg_m(\bx_\mu,\mu)=0$ for every $\mu \in J$ and  the pair of complex conjugated eigenvalues $\eta(\mu)\pm i \zeta(\mu)$ of  $D_\bx\bg_m(\bx_\mu,\mu)$ satisfies $\eta(\mu_0)=0$ and $\zeta(\mu_0)=\omega_0>0.$  }
\end{itemize}

From {\bf H1} (see \cite[Lemma 3]{ITCanNov2018}) we get the existence of a neighborhood $J_0\subset J$ of $\mu_0,$  a parameter $\e_1,$ $0<\e_1<\e_0,$ and a unique function $\bxi: J_0\times(-\e_1,\e_1)\rightarrow \R^2$ satisfying 
\begin{equation}\label{fixedP}
\bxi(\mu,0)=\bx_{\mu} \text{ and }\Pi(\bxi(\mu,\e),\mu,\e)=\bxi(\mu,\e), \text{ for every } (\mu,\e)\in J_0\times (-\e_1,\e_1).
\end{equation}
Consequently, for every $(\mu,\e)\in J_0\times(-\e_1,\e_1)$ the differential equation \eqref{tm1} admits a unique $T$-periodic orbit $\f(t,\mu,\e)$ satisfying $\f(0,\mu,\e)\to \bx_{\mu}$ as $\e\to0.$ We notice that, when  the differential equation \eqref{tm1} is defined in the extended phase space $\s^1\times\Omega,$  such a periodic solution is given by $\Phi(t,\mu,\e)=(t,\f(t,\mu,\e)).$

\smallskip

We also assume the following transversal hypothesis:

\smallskip

\begin{itemize}
\item[{\bf H2.}]
$
d=\displaystyle \dfrac{d\eta}{d\mu}(\mu_0)\neq 0.
$
\end{itemize} 

\smallskip

For each $(\mu,\e)\in J_0\times(-\e_1,\e_1)$, let $\lambda(\mu,\e)$ and $\ov{\lambda(\mu,\e)}$ be the pair of complex eigenvalues of $D_{\bz} \Pi(\xi(\mu,\e),\mu,\e).$ From {\bf H2} (see \cite[Lemma 4]{ITCanNov2018}),  we get the existence of $\e_2,$ $0<\e_2<\e_1,$ and a unique smooth function $\mu:(-\e_2,\e_2)\rightarrow J_0,$ with  $\mu(0)=\mu_0,$ satisfying
\begin{equation}\label{transH}
|\lambda(\mu(\e),\e)|=1,\,\, \big(\lambda(\mu({\e}),\e)\big)^k\neq 1,\,\, \text{ for } k\in\{1,2,3,4\},\,\, \text{ and } \dfrac{d}{d\mu} |\lambda(\mu,\e)|\Big|_{\mu=\mu({\e})}\neq 0.
\end{equation}

Finally, in order to state our last hypothesis, we apply the following change of variables and parameters  $\bx=\by+\xi(\mu,\e)$ and $\mu=\sigma+\mu(\e)$ to the the Poincar\'e map \eqref{quad2}, obtaining
\begin{equation}\label{sish2}
\begin{split}
\by\mapsto H_\e(\by,\sigma)&=\left( H^1_\e(\by,\sigma),H^2_\e(\by,\sigma)\right)\\
&=\Pi\left(\by+\bxi\left(\sigma+\mu(\e),\e\right),\sigma+\mu(\e),\e\right).
\end{split}
\end{equation}

In order to detect a Neimark-Sacker bifurcation in the map \eqref{sish2}, one still has to compute the first {\it Lyapunov Coefficient} of  \eqref{sish2} at  $(\by,\sigma)=(0,0).$ For that, consider the expansion of $D_{\by} H_{\e}(0,0)$ around $\e=0$,\begin{equation*}\label{dh}
D_{\by} H_{\e}(0,0)=Id+\e^m \mathcal{A}_\e+\CO(\e^{k+1}),
\end{equation*}
and assume the following technical hypothesis:
\begin{itemize}
\item[{\bf H3.}]  $Id+\e^m \mathcal{A}_\e$ is in its real Jordan normal form. More specifically,
$$
Id+\e^m \mathcal{A}_\e=\begin{pmatrix}
1+\tilde\eta(\e) & -\tilde\zeta(\e) \\
 \tilde\zeta(\e) & 1+\tilde\eta(\e)
\end{pmatrix},
$$
where 
$$\tilde\eta(\e)=\sum^k_{j=m}\e^j\eta_j\quad \text{and}\quad  \tilde\zeta(\e)=\sum^k_{j=m}\e^j\zeta_j,$$ with $\eta_i,\zeta_i\in \R$ for $i=1,\dots,k$.
\end{itemize}
In addition, define
\begin{equation}\label{ell}
\begin{split}
\ell_1^{\e}=&-\textrm{Re}\left(\dfrac{(1-2 e^{i\theta_\e})e^{-2i\theta_\e}}{2(1-e^{i\theta_\e})}\left\langle {\bf p},  B_\e({\bf p},{\bf p}) \right\rangle \left\langle {\bf p}, B_\e({\bf p},\ov {\bf p}) \right\rangle\right) \\
&+\textrm{Re}\left(e^{-i\theta_\e}\left\langle {\bf p}, C_\e({\bf p}, {\bf p}, \ov {\bf p})\right\rangle\right) -\dfrac{|\left\langle {\bf p},  B_\e({\bf p},\ov{\bf p})\right\rangle|^2}{2}-\dfrac{|\left\langle {\bf p},   B_\e(\ov{\bf p},\ov{\bf p})\right\rangle|^2}{4},
\end{split}
\end{equation}
where ${\bf p}=(1,-i)/\sqrt{2},$  $$
e^{i\theta_\e}=\lambda(\mu(\e),\e)=1+\sum^k_{j=m}\e^j (\eta_j+i \zeta_i)+\CO(\e^{k+1}),$$ and
\begin{equation}\label{multF}
B_\e(\bu,\bv)=\left(B_\e^1(\bu,\bv),B_\e^2(\bu,\bv)\right)\quad \mbox{and} \quad C_\e(\bu,\bv,\bw)=\left(C_\e^1(\bu,\bv,\bw), C_\e^2(\bu,\bv,\bw)\right)
\end{equation}
are multi-linear functions with the following components
$$
 B_\e^i(\bu,\bv)=\sum_{j,l=1}^2\dfrac{\partial^2
 H^i_\e}{\partial y_j\partial y_l}(0,0)\, u_jv_l\quad \mbox{and} \quad
C_\e^i(\bu,\bv,\bw)=\sum_{j,l,s=1}^2\dfrac{\partial^3H^i_\e}{\partial y_j\partial y_l\partial y_s}(0,0)\, u_j v_l w_s,
$$
 respectively. 
 
Accordingly, under hypotheses {\bf H1}, {\bf H2}, and {\bf H3}, the $k$-jet of the first Lyapunov coefficient of the map \eqref{sish2} at $(\by,\sigma)=(0,0)$ can be obtained by expanding $\ell_1^{\e}$ around $\e=0$:
\begin{equation}\label{Tel1}
\ell_1^{\e}=\e^m\ell_{1,m}+\e^{m+1}\ell_{1,m+1}+\e^{m+2}\ell_{1,m+2}+\cdots+\e^{k}\ell_{1,k}+\CO(\e^{k+1}),
\end{equation}
See \cite{ITCanNov2018} for explicit formulae of $\ell_{1,i},$ $m\leq i\leq k.$

The next result was stated in \cite{ITCanNov2018} assuming $d>0$ (see {\bf H2}). Here, we state the version of such a result for $d\neq0.$

\begin{theorem}[\cite{ITCanNov2018}]\label{teo2}
Let $m,$ $1\leq m\leq k,$ be the subindex of the first non-vanishing averaging function and let $\ell_{1,j},$ $m\leq j\leq k,$ as defined in \eqref{Tel1}. In addition to hypotheses {\bf H1}, {\bf H2}, and {\bf H3}, assume that $\ell_{1,j}\neq0$ for some $m\leq j\leq k.$ Let $j^*,$ $m\leq j^*\leq k,$ be the first subindex such that $\ell_{1,j^*}\neq0.$ Then, for each $\e>0$ sufficiently small there exist a $C^1$ curve $\mu(\e)\in J_0,$ with  $\mu(0)=\mu_0,$ and neighborhoods $\U_{\e}\subset \s^1\times \Omega$ of the periodic solution $\Phi(t,\mu(\e),\e)$ and  $J_{\e}\subset J_0$ of $\mu(\e)$ for which the following statements hold.\begin{itemize}

\item[$(i)$] For $\mu\in J_{\e}$ such that $\ell_{1,j^*}(\mu-\mu(\e))d\geq0,$ the periodic orbit $\Phi(t,\mu(\e),\e)$ is unstable (resp. asymptotically stable), provided that $\ell_{1,j^*}>0$ (resp. $\ell_{1,j^*}<0$), and the differential equation \eqref{tm1} does not admit any invariant tori in $\U_{\e}.$

\item[$(ii)$] For $\mu\in J_{\e}$ such that $\ell_{1,j^*}(\mu-\mu(\e))d<0,$ the differential equation \eqref{tm1} admits a unique invariant torus $T_{\mu,\e}$ in $\U_{\e}$ surrounding the periodic orbit $\Phi(t,\mu,\e).$ Moreover,  $T_{\mu,\e}$ is unstable (resp. asymptotically stable), whereas the periodic orbit $\Phi(t,\mu,\e)$ is asymptotically stable (resp. unstable), provided that $\ell_{1,j^*}>0$ (resp. $\ell_{1,j^*}<0$). 

\item[$(iii)$] $T_{\mu,\e}$ is the unique invariant torus of the differential equation \eqref{tm1} bifurcating from the periodic orbit $\Phi(t,\mu(\e),\e)$ in $\U_{\e}$ when $\mu$ passes through $\mu(\e).$
\end{itemize}
\end{theorem} 

The next result provides sufficient conditions for the existence of an invariant torus surrounding the periodic solution $\f(t,\e)$ given by Proposition \ref{p1} (see Figures \ref{fig3a} and \ref{fig3b}). We shall see that the parameter $\gamma_1$ will play the role of $\mu$. Hence, we denote $\f(t,\gamma_1,\e)=\f(t,\e).$  

\begin{proposition}\label{p5} Consider $\ell_1$ as defined in \eqref{CA}
 and assume that the R\"ossler System \eqref{sr} satisfies the hypotheses of Proposition \ref{p1}. If $\ell_1\neq 0$, then there exist a smooth curve $\gamma(\e),$ defined for $\e>0$ sufficiently small and satisfying $\gamma(\e)=\ov\gamma_1+\CO(\e)$ with $\ov\gamma_1=\al_1-\ov a \beta_1$, and intervals $J_{\e}$ containing $\gamma(\e)$ such that a unique invariant torus bifurcates from the periodic solution $\varphi(t,\gamma(\e),\e)$ as $\gamma_1$ passes through $\gamma(\e).$ Such a torus exists whenever $\gamma_1\in J_{\e}$ and $\ell_1(\gamma_1-\gamma(\e))>0,$ and surrounds the periodic solution $\f(t,\gamma_1,\e).$ In addition,  if $\ell_1>0$ (resp. $\ell_1<0$) such a torus is unstable (resp. asymptotically stable), whereas the periodic solution $\varphi(t,\gamma_1,\e)$ is asymptotically stable (resp. unstable).
\end{proposition}

\begin{proof} 
Consider the periodic differential system \eqref{CAs} and its first averaged function $\bg_1(\bz,\gamma_1)$ as given in \eqref{avfCA}. Taking $\bz=(r,z)$, we compute the second averaged function  $\bg_2(\bz,\gamma_1)=\left(\bg_2^1(\bz,\gamma_1),\bg_2^2(\bz,\gamma_1) \right)$ where
\begin{align*}
\bg_2^1(\bz,\gamma_1)&=\Big(-24 \pi  \ov{a}^2 \left(\ov{a}^2-2\right)^2 \left(\ov{a}^2+4\right) z^3-24 \pi  \ov{a} \left(\ov{a}^2-2\right)^2 z^2 \left(\alpha_1 \left(\ov{a}^2+1\right)^2\right.\\
&\left.+\ov{a} \left(\beta_1 \left(\ov{a}^2+7\right)-2 \ov{a} \gamma_1\right)-5 \gamma_1\right)-24 \pi  \left(\ov{a}^2-2\right)^2 z \left(\alpha_1 \left(\ov{a}^2-1\right)+\beta_1 \ov{a}-\gamma_1\right)\\
& \cdot\left(\alpha_1+\alpha_1 \ov{a}^2-\left(\ov{a}^2+1\right) \gamma_1+3 \beta_1 \ov{a}\right)+3 \pi  \ov{a} \left(-4 \ov{a}^6+8 \ov{a}^4+\ov{a}^2-2\right) r^3+\sqrt{2-\ov{a}^2} \\
&\cdot\left(r \left(-12 \pi ^2 \left(\ov{a}^2-2\right) \left(\alpha_1 \left(\ov{a}^2-1\right)+\ov{a} (\beta_1-\ov{a} \gamma_1)+\gamma_1\right)^2-12 \pi ^2 \left(\ov{a}^2-2\right)^2 \ov{a}^4 z^2\right.\right.
\\
&\left.\left.-24 \pi ^2 \left(\ov{a}^2-2\right)^2 \ov{a}^3 z (\alpha_1-\gamma_1)\right)-12 \pi ^2 \ov{a}^4 r^3\right)+r^2 \left(-16 \pi  \left(\ov{a}^2-2\right) \left(2 \ov{a}^4-3\right) \right.
\\
&\left.\cdot\ov{a}^2 z-8 \pi  \left(\ov{a}^2-2\right) \ov{a} \left(-4 \ov{a}^4 \gamma_1+4 \beta_1 \ov{a}^3+4 \ov{a}^2 \gamma_1+\alpha_1 \left(\ov{a}^4-\ov{a}^2+3\right)-3 \beta_1 \ov{a}+3 \gamma_1\right)\right)
\\
&+r \left(12 \pi  \ov{a} \left(\ov{a}^4+8 \ov{a}^2+1\right) \left(\ov{a}^2-2\right)^2 z^2+12 \pi  \left(\ov{a}^2-2\right)^2 \left(4 \alpha_2-\beta_1 (3 \alpha_1+\gamma_1)\right.\right.
\\
&+2 \ov{a}^4 (\alpha_2-\gamma_2)+\ov{a}^3 \left(-\alpha_1^2+2 \beta_2+\gamma_1^2\right)+\ov{a}^2 (-6 \alpha_2-4 \beta_1 \gamma_1+6 \gamma_2)+\ov{a} \left(3 \beta_1^2\right.
\\
&\left.\left.-\alpha_1^2-4 \beta_2+\gamma_1^2\right)-4 \gamma_2\right)-12 \pi  \left(\ov{a}^2-2\right)^2 z \left(2 \alpha_1-7 \beta_1 \ov{a}^3+3 \alpha_1 \ov{a}^2+\left(2 \ov{a}^4\right.\right.
\\
&+\left.\left.\left. 9 \ov{a}^2-2\right) \gamma_1-4 \beta_1 \ov{a}\right)\right)\Big)\dfrac{1}{24 \sqrt{2-\ov{a}^2} \left(\ov{a}^2-2\right)^4},
\\
\bg_2^2(\bz,\gamma_1)&=\Big(10 \pi  \ov{a}^4 r^3+\sqrt{2-\ov{a}^2} \left(r^2 \left(-6 \pi ^2 \left(\ov{a}^2-2\right) \ov{a} (\alpha_1-\gamma_1)-6 \pi ^2 \left(\ov{a}^2-2\right) \ov{a}^2 z\right)\right.
\\
&-24 \pi ^2 \ov{a}^2 \left(\ov{a}^2-2\right) z^3-36 \pi ^2 \ov{a} \left(\ov{a}^2-2\right) z^2 (\alpha_1+\beta_1 \ov{a}-\gamma_1)-12 \pi ^2 \left(\ov{a}^2-2\right) 
\\
&\left. \cdot z (\alpha_1+\beta_1 \ov{a}-\gamma_1)^2\right)+6 \pi  \left(\ov{a}^2-2\right)^2 z^2 \left(\alpha_1-3 \beta_1 \left(\ov{a}^2+2\right) \ov{a}+6 \ov{a}^2 \gamma_1+3 \gamma_1\right)
\\
&-18 \pi  \left(\ov{a}^6-3 \ov{a}^4+4\right) \ov{a} z^3-6 \pi  \left(\ov{a}^2-2\right)^2 z \left(2 \beta_2 \ov{a}^3-\beta_1 (\alpha_1+3 \gamma_1)-4 \alpha_2\right.
\\
&\left.+\ov{a}^2 (-\beta_1 (\alpha_1+3 \gamma_1)+2 \alpha_2-2 \gamma_2)+\ov{a} \left(-3 \alpha_1^2+3 \beta_1^2-4 \beta_2+3 \gamma_1^2\right)+4 \gamma_2\right)
\\
&+r \left(6 \pi  \left(\ov{a}^2-2\right) \left(\alpha_1+\alpha_1 \ov{a}^2-\left(\ov{a}^2+1\right) \gamma_1+3 \beta_1 \ov{a}\right) (\alpha_1+\ov{a} (\beta_1-\ov{a} \gamma_1)+\gamma_1)\right.
\\
&\left.+6 \pi  \ov{a}^2 \left(\ov{a}^6-6 \ov{a}^2+4\right) z^2+6 \pi  \ov{a} \left(\ov{a}^2-2\right) z \left(2 \ov{a} \left(\ov{a}^2+1\right) (\beta_1-\ov{a} \gamma_1) \right.\right.
\\
&\left.\left.+\alpha_1 \left(\ov{a}^4-2 \ov{a}^2+6\right)+6 \gamma_1\right)\right)+r^2 \left(3 \pi  \ov{a} \left(2 \ov{a}^6+\ov{a}^4-7 \ov{a}^2-6\right) z+3 \pi  \left(\ov{a}^2-2\right) \right.
\\
&\left. \cdot\left(4 \beta_1 \ov{a}^3+\alpha_1 \left(2 \ov{a}^4-5 \ov{a}^2-1\right)-\left(2 \ov{a}^4+\ov{a}^2+3\right) \gamma_1\right)\right)\Big)\dfrac{1}{6 \sqrt{2-\ov{a}^2} \left(\ov{a}^2-2\right)^4}.
\end{align*}

Assume that  $0<\ov{a}<\sqrt{2}$, $\mu=\gamma_1,$ $\mu_0=\ov\gamma_1=\al_1-\ov a \beta_1$,  and denoting the zero of $\bg_1$ by $\bx_{\gamma_1}=(\ov{r}_+,\ov{z}),$ we shall check  that the Poincar\'e map
$$
\Pi(\bz,\mu,\e)=\bz+\e\bg_1(\bz,\mu)+\e^2\bg_2(\bz,\mu) +\CO(\e^{3}),
$$
satisfies hypotheses {\bf H1}, {\bf H2}  and {\bf H3}. We point out that the following analysis would be the same assuming $-\sqrt{2}<\ov{a}<0$ and taking $\bz_{\gamma_1}=(\ov{r}_-,\ov{z}).$

In order to check hypothesis {\bf H1}, let
\[
 \omega_0=\dfrac{|\beta_1|  \ov{a}^2}{\left(2-\ov{a}^2\right)^{3/2}}.
 \]
We compute the characteristic polynomial of the Jacobian matrix $\dfrac{\partial\bg_1}{\partial\bz}(\bz_{\gamma_1},\gamma_1),$ obtaining
\[
\begin{split}
p(\lambda) =&
\lambda^2+\dfrac{d_1}{{\ov a}^2 \sqrt{2-{\ov a}^2}}  \lambda + \dfrac{d_0}{{\ov a}^2 \left(2-{\ov a}^2\right)^3},
\end{split}
\]
where $d_0$ and $d_1$ are defined in \eqref{CA}.  Denoting the roots of $p(\lambda)$ by $\lambda(\gamma_1)=\al(\gamma_1)\pm i\beta(\gamma_1)$, it is  straightforward to see that $\al(\ov\gamma_1)=0,$ $\beta(\ov\gamma_1)=\omega_0,$ and
$$
d=\dfrac{d\al(\ov\gamma_1)}{d \gamma_1}=\dfrac{-1}{2 \ov{a}^2 \sqrt{2-\ov{a}^2}}< 0,$$
which verifies hypotheses {\bf H1} and {\bf H2}.  

From hypotheses {\bf H1} and {\bf H2} and using the Implicit Function Theorem, we obtain the functions
\begin{equation*}
\bxi(\gamma_1,\e)=\bz_{\gamma_1}+\CO(\e)\quad \mbox{and} \quad \mu(\e)=\ov\gamma_1+\CO(\e)
\end{equation*}
satisfying \eqref{fixedP} and \eqref{transH}, respectively. Let $\bz=\bx+\bxi(\mu,\e)$, $\sigma=\gamma_1+\mu(\e)$, and  ${\bf A}=(a_{ij})\in\R^{2\times2}$ the matrix given by
\begin{align*}
a_{11}&=\e \left(\frac{\al_1 \left(\ov{a}^2+2\right)\omega_0}{\ov{a} \sqrt{4-2 \ov{a}^2}}+\frac{\left(4 \ov{a}^4-1\right) \sqrt{2-\ov{a}^2}\omega_0^2}{8 \pi  \ov{a}^3}-\frac{\omega_0^2}{2}\right),\\
a_{12}&=\sqrt{2-\ov{a}^2} \e \left(\frac{5 \ov{a}^2\omega_0^2}{6 \pi }-\frac{\al_1\omega_0+\ov{a} \omega_1}{\sqrt{2}}\right)-\frac{\ov{a} \sqrt{2-\ov{a}^2}\omega_0}{\sqrt{2}},\\
a_{21}&=\frac{\sqrt{2-\ov{a}^2} \e  \left(3 \sqrt{2} \pi  (\al_1\omega_0-\ov{a} \omega_1)-5 \ov{a}^2\omega_0^2\right)}{3 \pi  \ov{a}^2 \left(\ov{a}^2-2\right)}+\frac{\sqrt{2} \sqrt{2-\ov{a}^2}\omega_0}{2 \ov{a}-\ov{a}^3},\\
a_{22}&=\e  \left(\frac{1}{2}\omega_0 \left(-\frac{\al_1 \left(\ov{a}^2+2\right)}{\ov{a} \sqrt{1-\frac{\ov{a}^2}{2}}}-\omega_0\right)+\frac{\left(1-4 \ov{a}^4\right) \sqrt{2-\ov{a}^2}\omega_0^2}{8 \pi  \ov{a}^3}\right),\\
\omega_1&=\sqrt{2-\ov{a}^2} \left(\dfrac{\left(5-4 \ov{a}^2\right) \omega_0^2}{16 \pi  \ov{a}^2}-\dfrac{2 \pi  \beta_2 \ov{a}^2}{\left(\ov{a}^2-2\right)^2}\right)-\frac{\al_1 \left(\ov{a}^2+4\right) \omega_0}{\ov{a} \left(\ov{a}^2-2\right)}.
\end{align*}
 Taking the linear change of variables $\bx={\bf A}.\by$, we get the map
\begin{align*}
\by&\mapsto H_\e(\by,\sigma)={\bf A}^{-1}\Pi\left({\bf A}\by+\bxi\left(\sigma+\mu(\e),\e\right),\sigma+\mu(\e),\e\right),
\end{align*}
as defined in  \eqref{sish2}. Expanding the Jacobian matrix $D_{\by} H_{\e}(0,0)$ around $\e=0$
\begin{equation*}
D_{\by} H_{\e}(0,0)=\begin{pmatrix}
1&0\\
0&1
\end{pmatrix}+\e \begin{pmatrix}
0&-\omega_0\\
\omega_0&0
\end{pmatrix}+\e^2\begin{pmatrix}
-\dfrac{\omega_0}{2}&-\omega_1\\
\omega_1&-\dfrac{\omega_0}{2}
\end{pmatrix}+\CO(\e^3),
\end{equation*}
we see that it verifies hypothesis {\bf H3}. In order to obtain the Lyapunov Coefficient \eqref{ell}, we compute the multi-linear functions defined in \eqref{multF}:
\begin{align*}
 B_\e(\bu,\bv)=&\e\left( \frac{\pi  \ov{a} \sqrt{2-\ov{a}^2} \left(\ov{a}^2 u_2 v_2+4 u_1 v_1\right)}{\left(\ov{a}^2-2\right)^2},-\frac{\pi  \ov{a}^3 \sqrt{2-\ov{a}^2} (u_1 v_2+u_2 v_1)}{\left(\ov{a}^2-2\right)^2}\right)+\CO(\e^2),\\
 C_\e(\bu,\bv,\bw)=&\Bigg(\dfrac{-\e^2\pi  \ov{a}}{3 \sqrt{2-\ov{a}^2} \left(\ov{a}^2-2\right)^3} \Big(\left(\ov{a}^2-2\right) \left(\ov{a}^2+1\right) \Big(\left(2 \ov{a}^2+3\right) \ov{a}^2 u_1 v_2 w_2
 \\
&\left.+\left(2 \ov{a}^2+3\right) \ov{a}^2 u_2 (v_1 w_2+v_2 w_1)+36 u_1 v_1 w_1\right)+\ov{a} \sqrt{2-\ov{a}^2} 
\\
&\cdot\left(-2 \pi  \left(\ov{a}^2-2\right) \ov{a}^2 w_2 (u_1 v_2+u_2 v_1)-2 \pi  \left(\ov{a}^2-2\right) \ov{a}^2 u_2 v_2 w_1+\sqrt{2} \ov{a}\right.
\\
&\cdot\left(2 \left(\ov{a}^4+2 \ov{a}^2-2\right) u_1 (v_1 w_2+v_2 w_1)-5 \ov{a}^4 u_2 v_2 w_2+2 \left(\ov{a}^4+2 \ov{a}^2-2\right) u_2 v_1 w_1\right)
\\
& +48 \pi  u_1 v_1 w_1\Big)\Big),\dfrac{\e^2\pi  \ov{a} }{36 ( ba^2-2)^3}\Big(-3 \sqrt{2-\ov{a}^2} \big(8 \left(\ov{a}^4+8 \ov{a}^2+1\right) u_1 (v_1 w_2+v_2 w_1)
\\
&+8 \left(\ov{a}^4+8 \ov{a}^2+1\right) u_2 v_1 w_1+3 \left(4 \ov{a}^4-1\right) \ov{a}^2 u_2 v_2 w_2\big)-4 \Big(4 \sqrt{2} u_1 \big(9 \left(\ov{a}^2+4\right) v_1 w_1
\\
&+\left(3-2 \ov{a}^4\right) \ov{a}^2 v_2 w_2\big)+6 \pi  \left(\ov{a}^2-2\right) \ov{a}^3 u_1 (v_1 w_2+v_2 w_1)+\ov{a}^2 u_2 \Big(4 \sqrt{2} \left(3-2 \ov{a}^4\right)
\\
&\cdot v_1 w_2+4 \sqrt{2} \left(3-2 \ov{a}^4\right) v_2 w_1-9 \pi  \ov{a}^3 v_2 w_2+6 \pi  \left(\ov{a}^2-2\right) \ov{a} v_1 w_1\Big)\Big)\Big) \Bigg)+\CO(\e^3).
\end{align*}

Finally, taking \eqref{Tel1} into account, we compute 
\[
\ell_{1,1}=0 \quad \text{and}\quad \ell_{1,2}=\dfrac{\pi}{192\,\ov a^2 (2-\ov a^2)^3}{\ell_1},
\]
where $\ell_{1}$ is defined in \eqref{CA}. Hence,
\begin{align*}
\ell^\e_{1}=\e^2 \ell_{1,2}+\CO(\e^3),
\end{align*}
and the proof of Theorem \ref{t1} follows by applying Theorem \ref{teo2} with $m=1$ and $j^*=k=2$.
\end{proof}

\section{Numerical Examples}\label{sec:ex}

In this section, we provide three numerical examples for which Theorems \ref{t1} and \ref{t2} apply. In Examples 1 and 2, Theorem \ref{t1} predicts the existence of an unstable invariant torus and an asymptotically stable invariant torus, respectively. In Example 3, Theorem \ref{t2} predicts the existence of an asymptotically stable periodic solution.

\subsection{Example 1} Assume the following values for the coefficients of the R\"{o}ssler System \eqref{sr} in {\bf Case A}:
\begin{equation}\label{example1}
\ov a=\dfrac{1}{2}, \,\, \al_1=2,\,\, \beta_1=2,\,\, \al_2=-2, \,\, \beta_2=-1\,\, \text{ and }\,\, \gamma_2=-1.
\end{equation}
Thus, we compute $\ov\gamma_1=1\,\, \text{ and }\,\, \ell_{1}\simeq 155.66>0.$
Taking $\gamma_1=\ov\gamma_1+10^{-3},$
we have  $\ell_1(\gamma_1-\gamma(\e))>0$ for $\e>0$ small.  Theorem \ref{t1} predicts, for $\e>0$ small enough, the existence of an asymptotically stable periodic solution of \eqref{sr} surrounded by an unstable invariant torus (see Figure \ref{fig3a}).

\begin{figure}[H]
	\begin{overpic}[width=8cm]{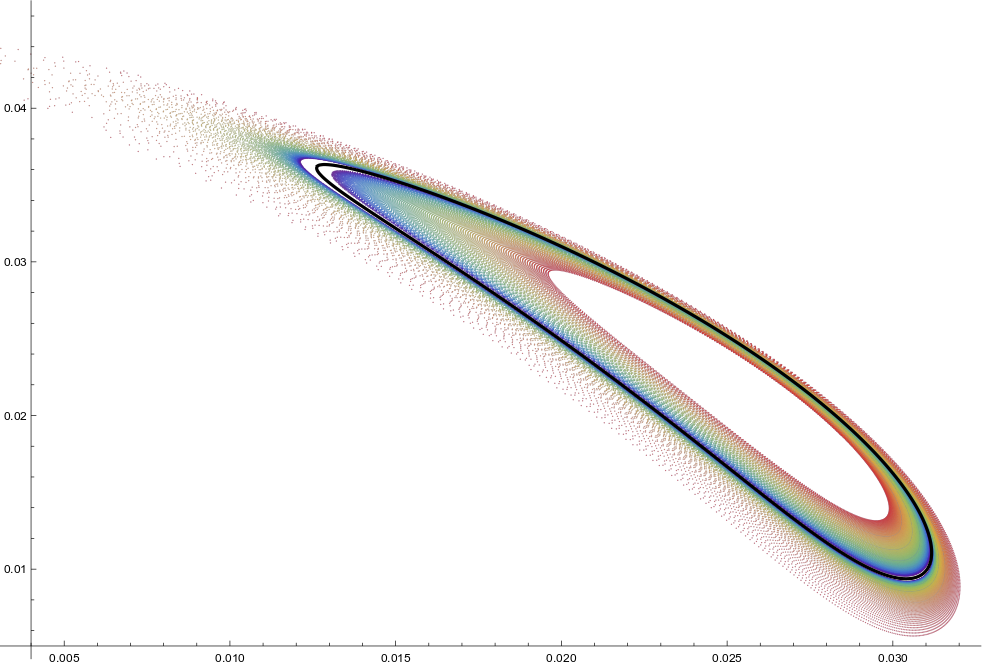}
	\end{overpic}
	\caption{Poincar\'{e} section $\{z=0\,\, \text{and}\,\, x>0\}$ of the R\"ossler System \eqref{sr}, assuming \eqref{example1} and $\e=10^{-2}.$ Trajectories starting at $(2982\times10^{-5},9656\times10^{-6},0)$ and $(3020\times10^{-5},9260\times10^{-6},0)$. 
The unstable invariant torus corresponds to an unstable invariant closed curve of the Poincar\'{e} map.}
	\label{fig3a}	
\end{figure}

\subsection{Example 2} Assume the following values for the coefficients of the R\"{o}ssler System \eqref{sr} in {\bf Case A}:
\begin{equation}\label{example2}
\ov a=-\dfrac{39}{32}, \,\, \al_1=1,\,\, \beta_1=2,\,\, \al_2=-2, \,\, \beta_2=-1,\,\, \text{ and }\,\, \gamma_2=-1.
\end{equation}
Thus, we compute $\ov\gamma_1=\dfrac{55}{16}$ and $\ell_{1}\simeq -1122.13<0.$ Taking  $\gamma_1=\ov\gamma_1-\dfrac{1}{250},$ we have  $\ell_1(\gamma_1-\gamma(\e))>0$ for $\e>0$ small. Theorem \ref{t1} predicts, for $\e>0$ small enough, the existence of an unstable  periodic solution of \eqref{sr} surrounded by an asymptotically stable invariant torus (see Figure \ref{fig3b}).

\begin{figure}[H]
	\begin{overpic}[width=9cm]{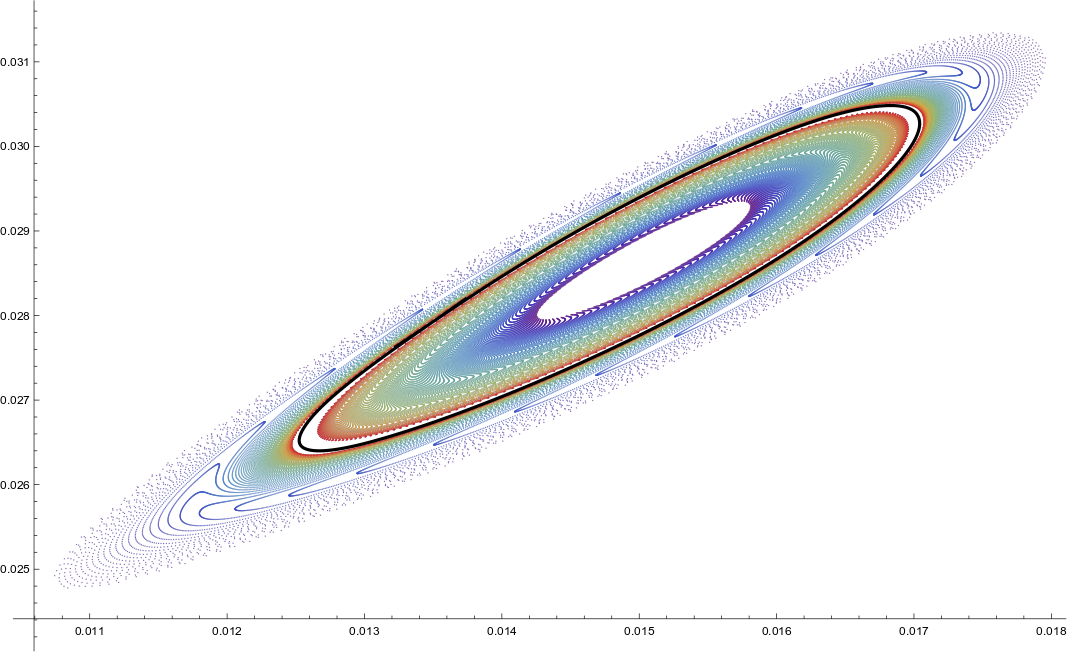}
	\end{overpic}
	\caption{
	Poincar\'{e} section $\{z=0\,\, \text{and}\,\, x>0\}$ of the R\"ossler System \eqref{sr}, assuming \eqref{example2} and $\e=10^{-2}.$ Trajectories starting at $(1504\times10^{-5},2852\times10^{-5},0)$ and $(1523\times10^{-5},2695\times10^{-5},0)$. 
The asymptotically stable invariant torus corresponds to an asymptotically stable invariant closed curve of the Poincar\'{e} map.}
	\label{fig3b}	
\end{figure}

\subsection{Example 3} Assume the following values for the coefficients of the R\"{o}ssler System \eqref{sr} in {\bf Case B}:
\begin{equation}\label{example3}
\begin{array}{l}
 \alpha_3=55,\,\, \alpha_4=\dfrac{37}{40},\,\, \alpha_5=\dfrac{57}{5},\vspace{0.2cm}\\
\beta_1=-1,\,\, \beta_2=-1,\,\, \beta_3=-\dfrac{177}{10},\,\, \beta_4=-1,\,\, \beta_5=18,\vspace{0.2cm}\\
\gamma_1=1,\,\, \gamma_2=-1,\,\, \gamma_3=0,\,\, \gamma_4=\dfrac{193}{10},\,\, \gamma_5=-\dfrac{247}{10},\,\, \text{and}\vspace{0.2cm}\\
\omega=\dfrac{39}{32}.
\end{array}
\end{equation}
Thus, we compute 
\[
 \alpha_1=\dfrac{497}{1024},\, \alpha_2=-\dfrac{1521}{1024},\,\lambda_1=-\dfrac{123}{239},\,\, \lambda_2=-\dfrac{116}{239},\,\, \text{ and }\,\, \delta=\dfrac{30963}{272}.
\]
Theorem \ref{t2} predicts, for $\e>0$ small enough, the existence of an asymptotically stable periodic solution of \eqref{sr}.

\begin{figure}[H]
	\subfigure[]
	{\begin{overpic}[width=5.5cm]{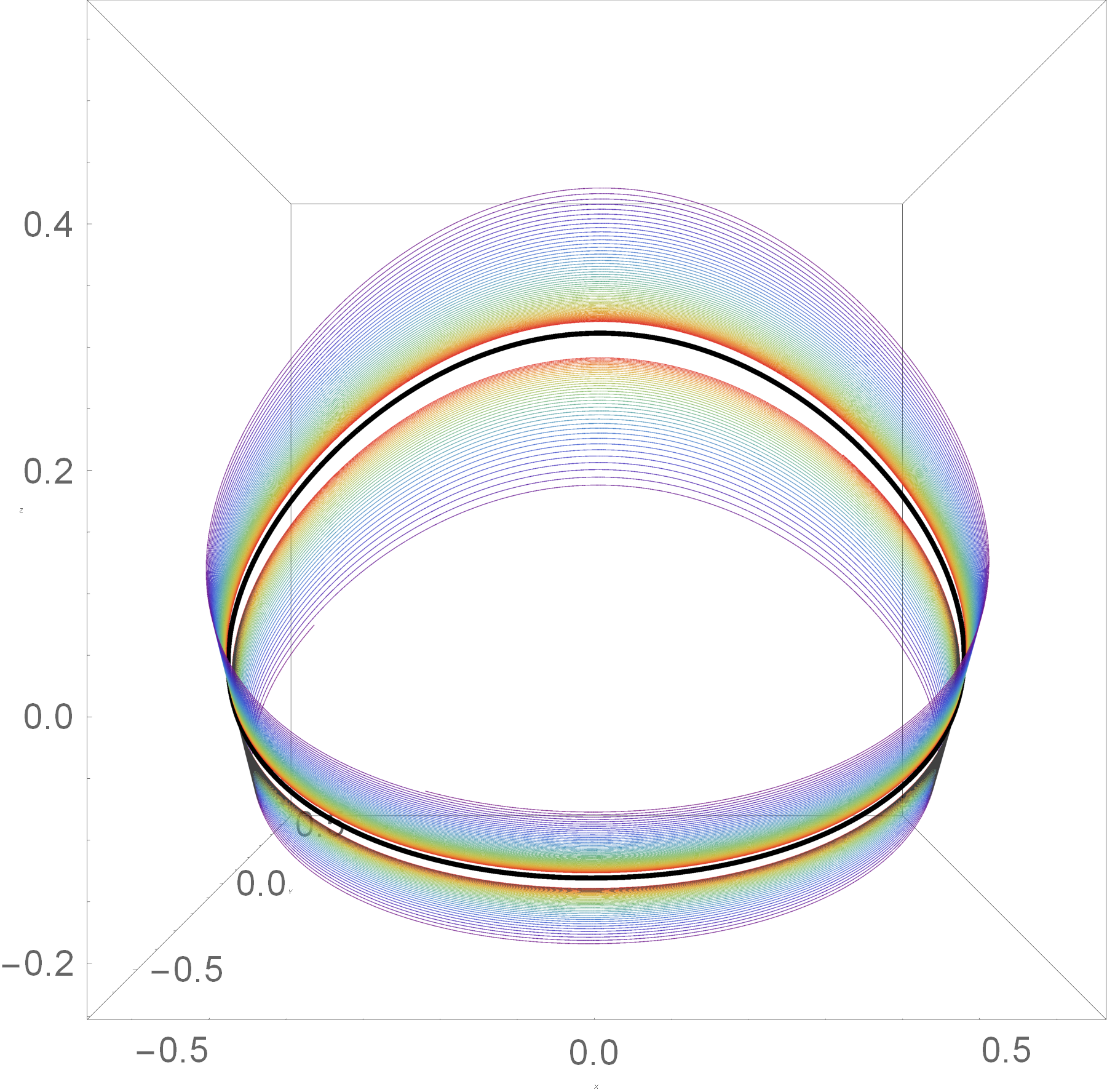}
	\end{overpic}}
		\subfigure[]
	{\begin{overpic}[width=6cm]{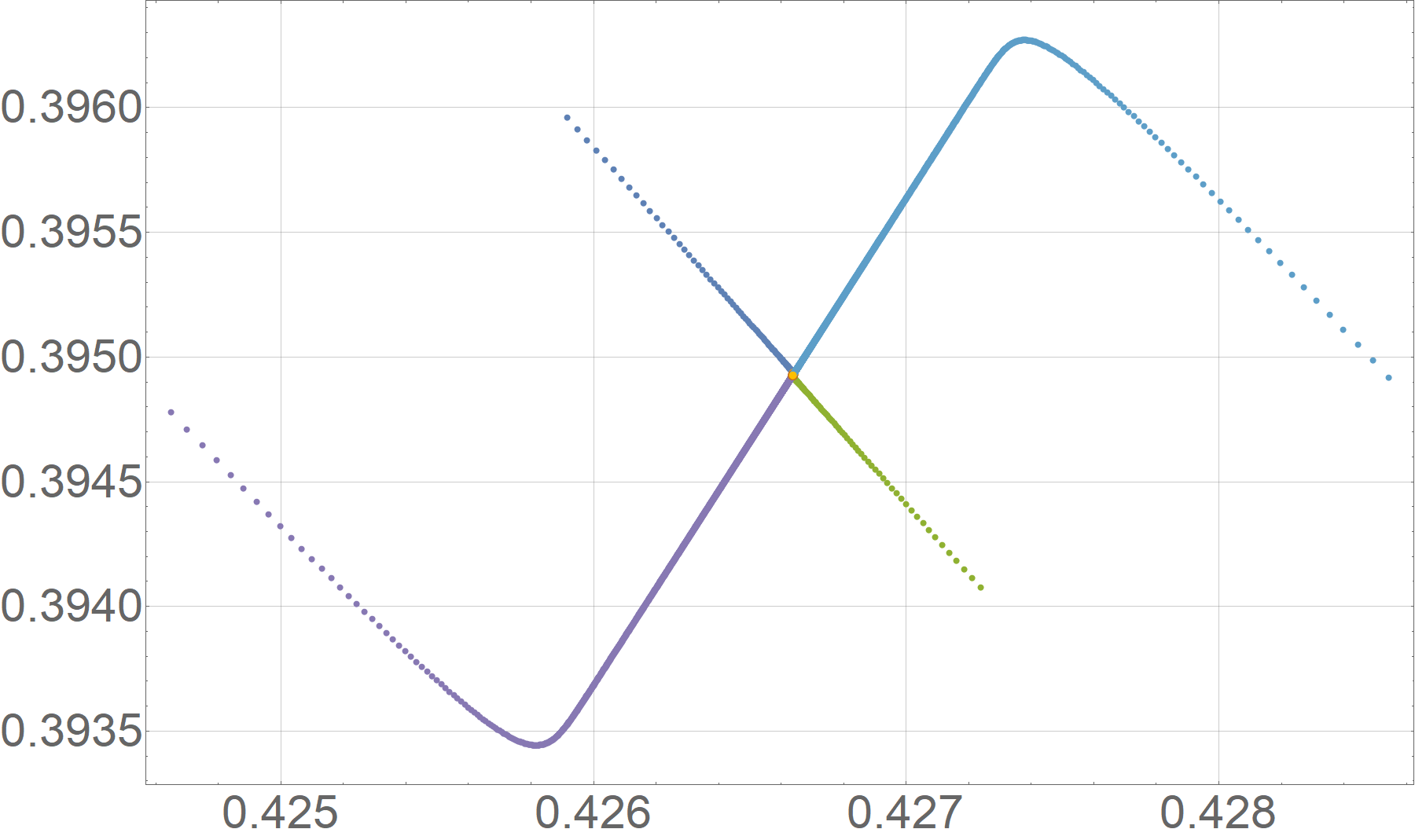}
	\end{overpic}}
	\caption{
	R\"ossler System \eqref{sr} assuming \eqref{example3} and $\e=1/50$.
Figure \ref{fig1}$(a)$ depicts two solutions being attracted by the periodic solution, represented by the closed curve. Figure \ref{fig1}$(b)$ depicts the intersection of four distinct solutions starting at $p_1=(0, 425\times 10^{-3}, 39725\times 10^{-5})$, $p_2=(0, 428\times 10^{-3}, 393\times 10^{-3})$, $p_3=\left(0, 4471/10530, 751/1902\right)$ and $p_4=(0, 4907/11449, 751/1902)$, with the Poincar\'{e} section $\{x=0 \text{ and } y>0\}$. The asymptotically stable periodic solution in Figure \ref{fig1}$(a)$ corresponds to an asymptotically stable fixed point in Figure \ref{fig1}$(b)$.}\label{fig1}
	\end{figure}

\section{Discussion}\label{sec:dis}

The R\"ossler System \eqref{sr}  is characterized by a three-parameter family of quadratic 3D vector fields and was introduced as a prototype of a simple autonomous differential system behaving chaotically for some values of the parameters. Studying the bifurcations occurring in the R\"ossler System has been a subject of interest for many authors. In our study, we were concerned about bifurcations of periodic solutions and invariant tori from zero-Hopf equilibria of the R\"ossler System. There exist two one-parameter families of R\"ossler Systems exhibiting a zero-Hopf equilibrium. Namely: {\bf Case  A} when $(a,b,c)=(\ov a, 1,\ov a)$, with $\ov a\in (-\sqrt{2},\sqrt{2})\setminus\{0\};$ and {\bf Case  B} when $(a,b,c)=(0, \ov b,0)$, with $\ov b\in (-1,\infty).$

For R\"ossler Systems near to the family of {\bf Case  A}, our main contribution (Theorem \ref{t1}) consisted in providing generic conditions ensuring the existence of a torus bifurcation. In this case, the torus surrounds a periodic solution that bifurcates from the zero-Hopf equilibrium.  Our analysis was based in a recent result \cite{ITCanNov2018} for detecting torus bifurcation through averaging theory. This kind of bifurcation had been previously indicated for the  R\"ossler System. Nevertheless, to the best of our knowledge, this is the first time that analytic generic conditions were provided ensuring  the existence of an invariant torus bifurcating from a zero-Hopf equilibrium in the R\"ossler System. The strategy followed by \cite{ITCanNov2018} consisted in looking for conditions on the averaged functions that ensure a {\it Neimark-Sacker Bifurcation} in the {\it Poincar\'{e} map}. This has been proven to be an effective method to detect torus bifurcation in 3D vector fields having zero-Hopf equilibria.

For R\"ossler Systems near to the family of {\bf Case  B}, the first-order averaging method had already been proven to not be able to detect any periodic solution bifurcating from the zero-Hopf equilibrium. Here, we showed that up to order five the usual recursive higher order averaging method does not provide any information about the existence of periodic solutions as well. This essentially means that the averaged functions, associated with the R\"{o}ssler System, do not have simple zeros. Thus, based on a recent result on averaging theory \cite{CLN}, our main contribution (Theorem \ref{t2}) consisted in providing generic conditions for the existence of a periodic solution bifurcating from the zero-Hopf equilibrium. This improved currently known results for such a family. The analysis performed in \cite{CLN} uses Lyapunov-Schmidt reduction to study the existence of periodic solutions bifurcating from non-isolated zeros of the first-order averaged function. Basically, this allowed us to use the second- and third-order averaged functions to perturb such a set of non-isolated zeros for obtaining sufficient conditions for the existence of a periodic solution. Theorem \ref{t2}  emphasizes the importance of the method developed in \cite{CLN}, which can improve the analysis of other systems through averaging theory.

In addition, the stability properties of such periodic solutions and invariant torus were analyzed. We showed that the periodic solution provided in {\bf Case A} can have its stability easily determined by the Jacobian matrix of the first-order averaged function, which in this case is hyperbolic. However, determining the stability of the periodic solution provided in {\bf Case B} required a more refined analysis, because in this case the Jacobian matrix of the first-order averaged function is not hyperbolic.
Accordingly, the theory of $k$-determined hyperbolic matrices \cite{Mu} was employed in order to use the forth- and fifth- averaged functions to study the stability of such a periodic solution. This procedure can be used to study the stability of periodic solutions obtained through averaging theory, however a general higher order approach is up to be developed.

\section*{Acknowledgements}

We thank the referees for their comments and suggestions which helped us to improve the presentation of this paper.

\smallskip

The authors thank Espa\c{c}o da Escrita - Pr\'{o}-Reitoria de Pesquisa - UNICAMP for the language services provided.

\smallskip

MRC is partially supported by FAPESP grants 2018/07344-0 and 2019/05657-4. DDN is partially supported by FAPESP grants 2018/16430-8, 2018/ 13481-0, and 2019/10269-3, and by CNPq grants 306649/2018-7 and 438975/ 2018-9. CV is partially supported by FCT/Portugal through UID/MAT/ 04459/2013.

\bibliography{CanNovVal2019}{}
\bibliographystyle{abbrv}

\end{document}